\documentclass[a4paper,10pt]{article}
 
\usepackage{amsmath}
\usepackage[a4paper,left=2cm,right=2cm,top=2cm,bottom=3cm]{geometry}
\usepackage{amsthm}
\usepackage{bm}
\usepackage{amssymb}
\usepackage[utf8]{inputenc}
\usepackage{authblk}
\usepackage{graphicx}
\usepackage{color}
\usepackage{tikz}
\usepackage{pgfplots}
\usepackage[ruled]{algorithm2e}
\usepackage{tabularx}
\usepackage{url}
\usepackage{xparse}
\usepackage{multirow}
\usepackage{dsfont}
\pgfplotsset{compat=1.13}
\usepackage{mathtools}
\usepackage[outline]{contour}
\usepackage[multiple]{footmisc}
\usepackage[normalem]{ulem}
\usepackage{mathrsfs}

\theoremstyle{plain}
\newtheorem{thrm}{Theorem}[section]
\newtheorem{lmm}[thrm]{Lemma}
\newtheorem{crllr}[thrm]{Corollary}
\newtheorem{prpstn}[thrm]{Proposition}

\theoremstyle{definition}
\newtheorem{dfntn}[thrm]{Definition}

\newtheorem{rmrk}[thrm]{Remark}

\theoremstyle{plain}

\newcommand{\Id}{\mathds{1}}

\newcommand{\R}{\mathbb{R}}
\newcommand{\N}{\mathbb{N}}

\renewcommand{\d}{\,\mathrm{d}}

\newcommand{\Dual}{\Borelsets}
\newcommand{\Leb}{\mathcal{L}}
\newcommand{\Count}{\mathcal{N}}
\newcommand{\Lagrangian}{\mathscr{L}}
\newcommand{\Wass}{\mathcal{W}}

\newcommand{\Prob}{\mathcal{P}}

\newcommand{\supp}{{\text{supp}\,}}
\newcommand{\dist}{{\text{dist}\,}}

\newcommand{\domain}{D}

\newcommand{\diam}{\text{diam}}

\newcommand{\powerDiagC}{\mathcal{F}}
\newcommand{\Func}{\mathcal{F}}

\newcommand{\KL}{\text{KL}}
\newcommand{\Lip}{\text{Lip}}

\newcommand{\discreteMeasure}{\mathbf{m}}
\newcommand{\g}{\text{g}}

\newcommand{\y}{\mathbf{y}}
\newcommand{\powerDiagD}{\mathbf{F}}

\renewcommand{\d}{\,\mathrm{d}}
\newcommand{\Borelsets}{\mathcal{D}}
\newcommand{\LC}{\boldsymbol{L}}
\newcommand{\CLC}{\boldsymbol{L}^c}
\newcommand{\SLC}{\Lagrangian}
\newcommand{\SCLC}{\Lagrangian^c}

\newcommand{\Laguerre}{L}
\newcommand{\sitesSet}{\mathcal{X}_n}
\newcommand{\gVector}{\mathbf{g}}
\newcommand{\sitesVector}{\mathbf{X}}

\newcommand{\powerDiag}{\mathbf{L}}

\definecolor{blue}{rgb}{0.137255,0,0.862745}
\definecolor{red}{rgb}{0.862745,0.137255,0}
\definecolor{myOrange}{RGB}{255, 169, 87}
\definecolor{myGreen}{RGB}{180, 255, 162}
\definecolor{myGrey}{RGB}{187, 187, 187}
\definecolor{myDarkGrey}{RGB}{100, ,100, 100}

\makeatletter
\DeclareRobustCommand\onedot{\futurelet\@let@token\@onedot}
\def\@onedot{\ifx\@let@token.\else.\null\fi\xspace}

\makeatother

\newcommand{\notinclude}[1]{}

\newcommand{\JJ}[1]{{\color{red} [JJ:{#1}] }}

\allowdisplaybreaks
\numberwithin{equation}{section}

\begin{document}
\title{Entropy-Regularized Optimal Transport in Information Design}
\author[1]{Jorge Justiniano}
\author[2]{Andreas Kleiner}
\author[2]{Benny Moldovanu}
\author[1]{Martin Rumpf}
\author[3]{Philipp Strack}

\affil[1]{Institute for Numerical Simulation, University of Bonn, Endenicher Allee 60, 53115 Bonn, Germany \authorcr
  \tt jorge.justiniano@outlook.de, martin.rumpf@ins.uni-bonn.de}
\affil[2]{Institute for Microeconomics, University of Bonn, Adenauerallee 24-42, 53113 Bonn, Germany \authorcr
  \tt andreas.kleiner@uni-bonn.de, mold@uni-bonn.de}
\affil[3]{Department of Economics, Yale University, New Haven CT 06520-8268\authorcr
  \tt philipp.strack@yale.edu}

\date{\today}

\maketitle
\begin{abstract} 

In this paper, we explore a scenario where a sender provides an information policy and a receiver, upon observing a realization of this policy, decides whether to take a particular action, such as making a purchase. The sender's objective is to maximize her utility derived from the receiver's action, and she achieves this by careful selection of the information policy. Building on the work of Kleiner et al., our focus lies specifically on information policies that are associated with power diagram partitions of the underlying domain. To address this problem, we employ entropy-regularized optimal transport, which enables us to develop an efficient algorithm for finding the optimal solution. We present experimental numerical results that highlight the qualitative properties of the optimal configurations, providing valuable insights into their structure. 
Furthermore, we extend our numerical investigation to derive optimal information policies for monopolists dealing with multiple products, where the sender discloses information about product qualities.

\end{abstract}

\section{Introduction}\label{sec:intro}
Mechanism Design is the branch of Game Theory/Economics that analyzes the
design of optimal institutions (or so called \textit{game-forms}) governing
the interaction of a group composed of self-interested, strategic and
privately informed agents. An optimal mechanism needs to solve an
information aggregation problem and an incentive problem in order to achieve
a desired collective decision (cf. the introductory textbook \cite{Borgers15}). 
Major applications have been made to social
choice and voting, to market design e.g., auctions and matching, and to
contracting. Numerous Nobel Prizes have been awarded for both theoretical
and applied work in this field, e.g. for Myerson's work on optimal auctions \cite{Myerson81}.

In classical mechanism design analysis, the private information available to
the agents about the environment is exogenously given. A more recent branch
of inquiry, called \textit{Information Design}, takes the institution that
governs the agents' interaction as given, but endogenizes the information
structure: in turn, the latter is chosen in order to optimize some given
goal.

A large literature within information design has focused on the \textit{%
Bayesian Persuasion }problem. A particularly important class is that of moment Bayesian persuasion problems 
which have been studied by Kamenica and Gentzkow \cite{Kamenica}, Kolotilin~\cite{Kolotilin}
and Dworczak and Martini~\cite{dworczak}.
With few exceptions, the literature on these problems 
assumes that the setting is unidimensional.
In this paper we deal with the multidimensional case and propose an efficient algorithm to compute 
optimal information designs. 
\paragraph{Organization. }
This paper is organized as follows. 
In Section~\ref{sec:InformationDesign} the moment Bayesian persuasion model will be introduced. 
Using recent insights in the model it will be sufficient to optimize over the closure of Lipschitz-exposed points of
a subset of measures induced by information policies. 
A particularly important class of these measures turns out to be represented by Laguerre partitions.
Then,  in Section~\ref{sec:Opti} the optimization problem of moment Bayesian persuasion is 
transformed into an optimization over such Laguerre partitions
along with a relaxation strategy that guarantees the existence of optimal designs. 
To compute these optimal Laguerre partitions we recall in Section~\ref{sec:SD}  
the role of Laguerre diagrams in the context of semi-discrete optimal transport, and derive in Section~\ref{sec:entropy}
the associated entropy relaxed optimal transport formulation, which allows for a computationally efficient algorithm to optimize over Laguerre partions and thus solve the moment Bayesian persuasion problem numerically. In this context, we demonstrate the convergence of maximizers of the relaxed problems to a maximizer of the non-relaxed problem. 
To ensure computational reproducibility, the spacial discretization is presented in full detail in Section~\ref{sec:space}. 
Then, in Section~\ref{sec:numerics} we show some qualitative features of the algorithm, and finally in Section~\ref{sec:application}
we use our algorithm to compute optimal information policies for a multi-product monopolist.

\section{Information Design and the Moment Bayesian Persuasion Problem}\label{sec:InformationDesign}
In this section we will rigorously formulate the moment Bayesian persuasion problem.  
Let $(D,\Borelsets,\nu)$ be a probability space, where $D\subseteq \R^d$ is compact and convex, $\Borelsets$ is its Borel $\sigma-$algebra, and $\nu$ is a probability measure which is assumed to be absolutely continuous with respect to the Lebesgue measure $\Leb$.
For later usage, we assume throughout this paper that the Radon-Nikodym density $\frac{\d\nu}{\d \Leb}\in L^\infty(\Leb)$.  The \emph{state of the world} $\omega$ is a realization of a random variable that is distributed according to the \textit{prior} $\nu$.
An informed \emph{sender} wants to persuade an uninformed \emph{receiver} to take an action
that the sender prefers. 
The receiver's optimal action depends on her information
about the state, and initially she only knows that the state of the world is distributed according to $\nu$. The sender, who knows $\nu$ and
observes the realization of the state of the world $\omega$, may
reveal information to the receiver about the realized state. This revelation
may often be strategic if the goals of
sender and receiver are different. For example, if the state is represented
by a one-dimensional random variable, the sender could reveal to the
receiver if the realized state is above or below a certain threshold but provide no additional information --- this
may be better for the sender than revealing all information.
 
Specifically, the sender commits to an \textit{information policy} $(S,\pi)$, where $S$ is a measurable space, $\Prob(S)$ is the space of probability distributions on $S$, and $\pi:D\rightarrow \Prob(S)$. 
If the sender commits to information policy $(S,\pi)$ and if the realized state of the world is $\omega$, the receiver observes the information policy and the realization of a random variable with values in $S$ that is drawn according to the probability distribution $\pi(\omega)$.
The prior $\nu$ and the information policy $\pi$ induce a joint probability distribution $\gamma$ on $D\times S$ defined by $\gamma(E\times T)\coloneqq \int_E \pi(\omega)(T) \,\mathrm d\nu(\omega)$ for any measurable sets $E\subseteq D$, $T\subseteq S$.
After seeing a realization $s\in S$, the receiver updates her beliefs about the state of the world according to Bayes' rule to a \emph{posterior belief} given by $\gamma(\cdot\times\{s\}|D\times\{s\})\in\Prob(D)$, which is a regular conditional probability measure of $\gamma$. This posterior belief determines the receiver's optimal behaviour and implicitly the sender's payoff. The expected value of the posterior belief is called the $\emph{posterior mean}$; therefore, each information policy induces a probability distribution over posterior means.

As an example, consider the fully revealing information policy given by $S=(D,\Borelsets)$ and $\pi(\omega)=\delta_{\omega}$, where $\delta_{\omega}$ denotes the Dirac measure at $\omega$. Under the fully revealing information policy, the random variable observed by the receiver equals the realized state of the world with probability 1 and hence the receiver perfectly learns the state of the world. For an example of a partially informative information policy, let $\{B_1, \ldots, B_n\}$ be a partition of $D$, $S = \left(\{1, \ldots, n\}, 2^{\{1, \ldots, n\}}\right)$, and $\pi(\omega) = \delta_i$ if $\omega\in B_i$. Under this information policy, the receiver learns in which partition element the realized state lies, and updates her prior belief accordingly.

In the moment Bayesian persuasion model 
it is assumed that the sender's payoff from inducing a posterior belief depends only on the posterior mean, and is given by an upper semicontinuous function $\Phi:D\rightarrow \mathbb{R}$. Note that, in reduced form, this formulation allows for the sender's payoff to depend on an action taken by the receiver.
To characterize the distributions of posterior means 
that can be induced by some information policy we use the following stochastic dominance concept.

\begin{dfntn}
[Shaked and Shanthikumar \cite{shaked}, Cartier et al. \cite{Cartier}
and Phelps \cite{Phelps}] For measures $\nu $ and $\rho $, we say that $\nu $ 
\emph{dominates }$\rho$\emph{\ in the convex order} (or that $\rho \ $is a 
\emph{fusion} of $\nu ),\ $denoted by $\nu \succeq \rho $, if $\int \psi \,%
\mathrm{d}\nu \geq \int \psi \,\mathrm{d}\rho $ for all convex functions $%
\psi $ for which both integrals exist. We write $\nu \succ \rho $ if $\nu $
dominates $\rho $ in the convex order and $\nu \neq \rho .$\ We denote by $%
F_{\nu }=\{\rho :$ $\rho \preceq \nu \}$ the set of fusions of a given
measure $\nu .$
\end{dfntn}

Any information policy necessarily reveals weakly less information than the state of the world and the generated distribution of posterior means is dominated in the convex order by the prior. Conversely, for any probability measure $\rho \preceq \nu$ there exists an information policy that generates $\rho$ as its probability measure of posterior means (see Blackwell \cite{blackwell} and Strassen \cite{Strassen}). 

Therefore, instead of modeling the sender's choice of an information policy we can assume that the sender directly chooses a probability measure that is dominated in the convex order by $\nu$. 
The sender chooses such a probability measure to maximize her expected payoff, and therefore solves the problem 
\begin{align}\label{eq:optimize}
\max_{\lambda \in F_{\nu }}\int_D \Phi(y)\,\mathrm{d}\lambda (y).
\end{align}
This is the information design problem in the focus of this paper.

To deduce qualitative properties of solutions, let us recall some basic concepts of convex analysis.
An \textit{extreme} point of a convex set $A$ is a point $y\in A$ that
cannot be represented as a convex combination of two other points in $A$.
I.e. $y\in A$ is an extreme point of $A$ if $y=\alpha w+(1-\alpha )z,$
for $w,z\in A$ and $\alpha \in \lbrack 0,1]$ jointly imply that $y=w$ or $y=z$.

The \textit{Krein--Milman Theorem} states that any convex and compact
set $A$ in a locally convex space is the closed, convex hull of its extreme
points. In particular, such a set has extreme points. The usefulness of
extreme points for optimization stems from \textit{Bauer's Maximum Principle},
which states that a convex, upper semicontinuous functional on a non-empty, compact and
convex set $A$ of a locally convex space attains its maximum at an extreme
point of $A.$ 
An element $y$ of a convex set $A$ is called
\textit{exposed} if there exists a supporting hyperplane $H$ such
that $H\cap A=\{y\}$ or, equivalently, there is a continuous, linear functional that attains
its unique maximum on $A$ at $y$.
Every exposed point is extreme, but the converse is not true in general.

The set of fusions $F_{\nu }$ appearing in the above maximization problem is
convex and compact in the weak$^*$ topology of measures. As the objective is linear in the
measure $\lambda$, a maximum is attained at one of the extreme points of $%
F_{\nu }$. It is thus of interest to further explore the structure of
extreme and exposed points, and we focus here on those measures that have a
finite support.

\medskip

For any measure $\rho$ on $D$ and a measurable set $B\subseteq D$, we denote by 
$\rho\vert_B$ the restriction of $\rho$ to $B$.
As stated in the following theorem a key feature of any extremal measure in $F_{\nu }$ with finite support is that there is a partition 
of the domain $D$ into convex sets $B$ such that all the
original mass restricted to $B,\ \nu(B),$ remains within $B$ and is fused
into a measure $\rho |_{B}$ whose support is an affinely independent set of
points.

\begin{thrm}[Kleiner et al. \cite{Kleiner2}]
\label{prop:necessity_extreme} Let $%
D\subseteq \mathbb{R}^{n}$ be compact and convex, and let $\nu $ be an
absolutely continuous probability measure on $D$. Suppose that $\rho $ is an
extreme point of $F_{\nu }$ with finite support. Then there exists a finite
partition $\mathcal{P}$ of $D$ into convex sets such that, for each $B\in 
\mathcal{P}$, $\rho |_{B}\preceq \nu |_{B}$ and $\rho |_{B}$ has affinely
independent support.
\end{thrm}

The above result generalizes a unidimensional result found in Kleiner et al. \cite{Kleiner1} and 
Arieli et al. \cite{arieli}. In fact, more precise information is
available about the geometric structure of the following subset of extreme points.

\begin{dfntn}
A measure $\rho \in F_{\nu }$ is a \emph{Lipschitz-exposed} point 
of $F_{\nu }$ if there exists a Lipschitz-continuous function $%
\Phi\colon D\rightarrow \mathbb{R}$ such that $\rho $ is the unique solution to
the problem 
\begin{align*}
\max_{\lambda \in F_{\nu }}\int \Phi(y)\,\mathrm{d}\lambda (y).
\end{align*}
\end{dfntn}

The following theorem provides a characterization of Lipschitz-exposed points with finite support
using Laguerre diagrams, also known as \emph{power diagrams}, see Aurenhammer \cite{Au87}.
Given an $n$-tupel $\sitesVector=(x_1,\ldots,x_n)$ of pairwise distinct sites $x_i\in \R^d$
and a weight vector $\gVector =(\g_1, \ldots, \g_n) \in \R^n$, the associated Laguerre cells $\Laguerre_i[\sitesVector,\gVector]\in\Borelsets$ are the convex polyhedra 
\begin{align}\label{eq:LaguerreCell}
\Laguerre_i[\sitesVector,\gVector]\coloneqq\lbrace y\in\domain : \vert y-x_i\vert^2-\g_i\leq \vert y-x_j\vert^2-\g_j \ \ \forall \ \ 1\leq j\leq n\rbrace,
\end{align}
for $i=1,\ldots,n$. The ensemble of all Laguerre cells forms the Laguerre diagram. 

\begin{thrm}[Kleiner et al. \cite{Kleiner2}]
\label{prop:characterization_strongly_exposed} Let $D\subseteq \mathbb{R}%
^{n} $ be compact and convex, and let $\nu $ be a probability measure with
full support on $D$ that is absolutely continuous with respect to the
Lebesgue measure.
Let $\rho \in F_{\nu }$ have finite support. Then $\rho $ is a
Lipschitz-exposed point of $F_{\nu }$ if and only if 
there exists a Laguerre
diagram $\mathcal{P}$ of $D$ such that, for all probability measures $\lambda$, if  $\lambda |_{\Laguerre}\preceq \nu |_{\Laguerre}$ for
all $\Laguerre\in \mathcal{P}$ and $\supp(\lambda )\subseteq \supp(\rho )$ then $\lambda =\rho $.
\end{thrm}

In particular, $\rho \in F_{\nu }$ is 
a finitely supported Lipschitz-exposed point if there is a
Laguerre diagram such that for each Laguerre cell $\Laguerre$ with non-vanishing measure $\nu(\Laguerre),$ the relation $\rho |_{\Laguerre}\preceq \nu |_{\Laguerre}$ holds
and if the support of $\rho |_{\Laguerre}$ is affinely independent. 
Indeed, in this case $\lambda |_{\Laguerre}\preceq \nu |_{\Laguerre}$ implies that $\lambda |_{\Laguerre}$ and $\rho |_{\Laguerre}$ have the same mean. In turn, this implies that if $\supp(\lambda )\subseteq \supp(\rho )$ and if the support of $\rho$ is affinely independent then $\rho |_{\Laguerre} = \lambda |_{\Laguerre}$.

For any compact convex set in a normed vector space, the set of exposed points is dense in the set of extreme points of $F_{\nu}$ (Klee, \cite{klee}). Therefore, to optimize a continuous linear objective functional on such a set, it would be sufficient to optimize over the closure of exposed points. In the following, we will focus on the subset of exposed points that are Lipschitz-exposed and where, on each cell $\Laguerre$ of the corresponding Laguerre
diagram, the mass within the respective cell is fused to a unique mass point. 
These extreme points have received considerable attention in the Economics literature in one-dimensional settings (see for example Dworczak and Martini~\cite{dworczak} or Ivanov \cite{Ivanov}). They are simple to use in practice because they can be implemented by only revealing, for a partition of the state space into convex sets, in which partition element the realized state lies. Finally, the relation to Laguerre diagrams  makes these exposed points numerically tractable.
Our computational method derived below focuses on solving the
sender's problem among such exposed points.

\section{The optimization task}\label{sec:Opti}
Motivated by Theorem \ref{prop:characterization_strongly_exposed} we discuss in this section 
the optimization problem \eqref{eq:optimize} of moment Bayesian persuasion as an optimization over Laguerre partitions
along with a relaxation strategy. 
At first, we study the case where the sender commits to an information policy induced by general partitions. To this end, we consider a probability space $(\domain,\Borelsets,\nu)$ on a compact and convex domain $\domain\subset\R^d$ for $d\geq 2$, equipped with a probability measure $\nu\in\Prob(\domain)$ on the Borel-$\sigma$-algebra $\Borelsets$, which is assumed to be absolutely continuous with respect to the Lebesgue measure $\Leb$ with Radon-Nikodym density $\frac{\d\nu}{\d \Leb}\in L^\infty(\Leb)$. For fixed $n\in\N$, let $(B_i)_{i=1,\ldots, n}\subset\Borelsets$ be a $\nu$-partition of $\domain$, i.e. $\nu(B_i\cap B_j)=0$ for $i\neq j$, and $\nu\left(\cup_{i=1}^nB_i\right)=1$. Let $(S, \mathcal{S})\coloneqq \left(\{1, \ldots, n\}, 2^{\{1, \ldots, n\}}\right)$, and assume the sender commits to the information policy $\pi:D\to\Prob(S)$, with $\pi(\omega)=\delta_i\in \Prob(S)$ if and only if $\omega\in B_i$.
After receiving the signal realization $i\in S$, the receiver updates her belief to the posterior $\gamma(\cdot\times \{i\}\ |\ D\times\{i\})$. It holds that
$$\gamma(B\times \{i\}\ |\ D\times\{i\})=\frac{\int_B \pi(\omega)(\{i\}) \,\mathrm d\nu(\omega)}{\int_D \pi(\omega)(\{i\}) \,\mathrm d\nu(\omega)}=\frac{\int_B \Id_{B_i}(\omega) \,\mathrm d\nu(\omega)}{\int_D  \Id_{B_i}(\omega) \,\mathrm d\nu(\omega)}=\frac{\nu(B\cap B_i)}{\nu(B_i)}=\nu(B|B_i).$$
For any non-$\nu$-null set $B\in\Borelsets$, we define the $\nu$-barycenter  
\begin{align*}
b[B]\coloneqq\frac{\int_{B}y\,\d\nu(y)}{\nu(B)}.
\end{align*}
The $\nu$-barycenter of $B_i$ coincides then with the mean of the posterior $\nu(\ \cdot \ |B_i)$. 
To conclude, the posterior mean distribution generated by the information policy $\pi$ is given by $\sum_{i=1}^n\nu(B_i)\delta_{b[B_i]}\in F_\nu$. In this case, for a continuous function $\Phi:\domain \to \R$  the functional \eqref{eq:optimize} takes the following explicit form:
\begin{align}\label{eq:objective_og}
\sum_{i=1,\ldots, n \atop \nu(B_i)>0} \nu(B_i)\Phi(b[B_i]).
\end{align}

Now, in the light of Section~\ref{sec:InformationDesign} we focus on partitions described by Laguerre diagrams.
In this case, we may optimize the cost functional $\eqref{eq:objective_og}$ directly on the parameters $(\sitesVector,\gVector)$ describing Laguerre cells
\begin{align}\label{eq:objective}
\mathcal{F}_n[\sitesVector,\gVector]\coloneqq \sum_{i=1,\ldots, n \atop m_i[\sitesVector,\gVector]>0} m_i[\sitesVector,\gVector]\Phi(b_i[\sitesVector,\gVector]),
\end{align}
where $m_i[\sitesVector,\gVector]\coloneqq \nu(\Laguerre_i[\sitesVector,\gVector]),$ and $b_i[\sitesVector,\gVector]\coloneqq b[\Laguerre_i[\sitesVector,\gVector]]$ (cf. equation \eqref{eq:LaguerreCell}), to be maximized over $n$-tuples of pairwise distinct sites 
$
\sitesSet\coloneqq\lbrace \sitesVector=(x_1,\ldots,x_n)\in\left(\R^{d}\right)^n: x_i \in \R^{d},\;  x_i \neq x_j \ \text{ for } \ i\neq j \rbrace
$
 and weight vectors $\gVector \in \R^n$. 
Note that Laguerre cells might be sets of vanishing measure.
Since Laguerre cells are by definition a $\nu$-partition of $\domain$, the set of associated characteristic functions $\chi_{\Laguerre_i[\sitesVector,\gVector]}$ of the 
Laguerre cells $\Laguerre_i[\sitesVector,\gVector]$ for $i=1,\ldots, n$ forms a partition of unity of $\domain$, i.e.
 $\sum_{i_1,\ldots n} \chi_{\Laguerre_i[\sitesVector,\gVector]}=1$ a.e. in $\domain$. We define the power diagram associated with an $n$-tuple of sites $\sitesVector$ and a weight vector $\gVector$
as
\begin{align}
\LC[\sitesVector,\gVector] \coloneqq (\chi_{\Laguerre_1[\sitesVector,\gVector]},
\ldots,\chi_{\Laguerre_n[\sitesVector,\gVector]}).
\end{align} 
Along maximizing sequences for the functional $\mathcal{F}_n[\cdot, \cdot]$  
it might happen that subsets of Laguerre cells 
collapse, or that sites as well as weights diverge.
Here, the notion of power diagram as partitions of unity 
helps to deal with these degenerate cases. At first, we obtain
the following relative compactness result: 
\begin{lmm}\label{lem:relComp}
The set $\SLC \coloneqq \LC(\sitesSet\times\R^n)$ is relatively compact in 
$L^1(\nu)^n$ and any limit of a converging sequences in $\SLC$ is  a partition of unity a.e. on $\domain$. 
\end{lmm}
\begin{proof}
Let $\nu$ be trivially extended onto $\R^d$ with density $0$ outside $\domain$.  
Each non-empty Laguerre cell $\Laguerre_i[\sitesVector,\gVector]$ is convex and its boundary in the interior of 
$\domain$ is polygonal and consists of at most $n-1$ planar interfaces. 
Each of these interior interfaces has at most a $\mathcal{H}^{d-1}$ measure $\diam(\domain)^{d-1}$ inside of $\domain$. 
Thus, for $h\in \R$ we observe that  
$$
\Vert\chi_{\Laguerre_i[\sitesVector,\gVector]}(\cdot-h)-
\chi_{\Laguerre_i[\sitesVector,\gVector]}\Vert_{L^1(\nu)}\leq(n-1)\left\Vert\tfrac{\d\nu}{\d x}\right\Vert_\infty\diam(\domain)^{d-1}h\,.
$$ 
This, together with the compactness of $\domain$, implies by the Fr\'echet-Kolmogorov theorem the relative compactness of 
$\powerDiag(\sitesSet\times\R^n)$ in $L^1(\nu)$. Let $(\sitesVector^k)_k\subset\sitesSet$, and $(\gVector^k)_k\subset\R^n$.  
Since $\chi_{\Laguerre_i[\sitesVector^k,\gVector^k]}=0$ outside 
$\domain$ any $L^1(\nu)$-limit of $(\chi_{\Laguerre_i[\sitesVector^k,\gVector^k]})_{i=1,\ldots, n}$ for $k\to \infty$ 
is of the form $\left(\chi^1,\ldots,\chi^n\right) \in L^1(\nu;\{0,1\})^n$ and $\supp \chi_i \subseteq\domain$. 
Finally, 
$\Vert\sum_{i=1,\ldots, n} \chi^i-1\Vert_{L^1}=
\lim_{k\rightarrow\infty}\Vert\sum_{i=1,\ldots, n} \chi_{\Laguerre_i[\sitesVector^k,\gVector^k]}-1 \Vert_{L^1}=0$
and thus $\left(\chi^1,\ldots,\chi^n\right)$ is a partition of unity.
\end{proof}
Let us denote by $\SCLC$ the closure of $\SLC=\LC(\sitesSet\times\R^n) \in L^1(\nu;\{0,1\})^n$.
For $\CLC = \left(\chi^1,\ldots,\chi^n\right) \in \SCLC$ 
we define the relaxed functional 
\begin{align}\label{eq:relaxObjective}
\mathcal{F}_n^c[\CLC] \coloneqq \sum_{i=1,\ldots, n \atop m[\chi^i]>0} m[\chi^i] \; \Phi(b[\chi^i])
\end{align}
where we define with a slight misuse of notation the $\nu$-mass $m[\chi] \coloneqq \int_\domain \chi \d\nu$ of a characteristic function $\chi$ and its $\nu$-barycenter $b[\chi] \coloneqq m[\chi]^{-1} \int_\domain x\,\chi \d\nu$ for $m[\chi] >0$.
Let us remark that some of the $\chi^i$ in $\CLC$ (but not all) might have zero mass $m[\chi^i]$.
For this relaxed functional we obtain the following existence result of a maximum.
\begin{thrm}
Assume $\Phi$ to be a upper semicontinuous function on $\domain$. Then $\mathcal{F}_n^c$ attains its maximum on $\SCLC$.
\end{thrm}
\begin{proof}
At first, we recall that upper semicontinuous functions on compact domains are bounded from above. 
Let $\bar \Phi$ denote the maximum of $\Phi$ on $\domain$, which exists due to the upper semicontinuity of $\Phi$.
Let $$\left(\CLC_k=(\chi^1_k,\ldots,\chi^n_k)\right)_{k\in \N}\subset \SCLC$$ be a maximizing sequence of \eqref{eq:relaxObjective}. 
Due to the compactness of $\SCLC$, we obtain that, up to the selection of a subsequence, $(\chi^1_k,\ldots,\chi^n_k)$ converges in $L^1(\nu)$ 
to a limit $\CLC=(\chi^1,\ldots,\chi^n)$.
For all $i\in \{1,\ldots, n\}$ with $m[\chi^i]=0$ it holds that $m[\chi_k^i] \to 0$ for  $k\to \infty$ 
and thus $m[\chi_k^i]  \Phi(b[\chi^i]) \leq m[\chi_k^i] \bar \Phi \to 0$. 
For $m[\chi^i]>0$ the sequence of barycenters $\left(b[\chi^i_k]\right)_{k\in \N}$ converges to $b[\chi^i]$. 
Taking into account the upper semicontinuity of $\Phi$, this implies that 
$\limsup_{k\to \infty} \Phi(b[\chi^i_k]) \leq \Phi(b[\chi^i])$.
Altogether, we obtain 
\begin{align*}
\limsup_{k\to \infty} \mathcal{F}_n^c[\CLC_k] = \limsup_{k\to \infty} \sum_{i=1,\ldots, n \atop m[\chi^i_k]>0} m[\chi^i_k] \; \Phi(b[\chi^i_k]) 
\leq \sum_{i=1,\ldots, n \atop m[\chi^i]>0} m[\chi^i] \; \Phi(b[\chi^i]) =\mathcal{F}_n^c[\CLC]\,.
\end{align*}
Hence, the relaxed functional $\mathcal{F}_n^c$ attains its maximum on $\SCLC$ at $\CLC$.  
\end{proof}

\begin{rmrk}
The relation to the concept of stochastic dominance is as follows.
As an upper semicontinuous function, $\Phi$ is a $\nu$-measurable function on $\domain$.  
Given a polyhedral partition of $\domain$ into Laguerre cells 
$\left(\Laguerre_i[\sitesVector,\gVector]\right)_{i=1,\ldots, n}$,
one may collapse the mass of each cell at its barycenter. 
This induces an atomic probability measure 
$$\rho \coloneqq 
\sum_{i=1,\ldots, n \atop m_i[\sitesVector,\gVector]>0} m_i[\sitesVector,\gVector] \; \delta_{b_i[\sitesVector,\gVector]}
$$
on $\domain$.
Recall that a probability measure $\nu$ dominates a probability measure $\rho$ in convex order if and only if $\mathbb{E}_{X\sim\nu}[\Phi(X)]\geq \mathbb{E}_{X\sim\rho}[\Phi(X)]$ for all convex functions $\Phi:\R^d\rightarrow\R$ such that both expectations exist. 
By Jensen's inequality, this indeed holds for the initial 
probability measure $\nu$ and the atomic probability measure $\rho$ considered here.
Hence, in our ansatz we consider the expected value of the given function $\Phi$ with respect to 
an atomic measure associated with some power diagram $\CLC$ and maximize this expected value over all atomic measures induced by power diagrams $\CLC$. 
\end{rmrk}
To avoid the relaxation, one might consider hard constraints to ensure that 
pairwise distances between sites and cell masses do not vanish in the limit along a maximizing sequence.
Alternatively, a penalty formulation can be used as a more robust and effective alternative.
To this end, we define for a penalty parameter $\eta>0$ 
\begin{align}\label{eq:objective3}
\mathcal{F}_n^\eta(\sitesVector,\gVector)= \mathcal{F}_n(\sitesVector,\gVector) - \eta\mathcal{R}_n(\sitesVector,\gVector),
\end{align}
where the penalty term $\mathcal{R}_n$ is given as 
\begin{align}\label{def:regularizer}
\mathcal{R}_n(\sitesVector,\gVector)\coloneqq\sum_{i=1,\ldots, n \atop m_i[\sitesVector,\gVector]>0}\int_\domain\vert y-x_i\vert^2\chi_{\Laguerre_i[\sitesVector,\gVector]}(y)\d\nu(y)+\sum_{1\leq i,j \leq n \atop {i\neq j\atop m_i[\sitesVector,\gVector],m_j[\sitesVector,\gVector]>0}}\frac{m_i[\sitesVector,\gVector]m_j[\sitesVector,\gVector]}{\vert x_i-x_j\vert^2}.
\end{align}
Here, the cell masses as scaling factors and the characteristic functions of cells as weight functions
are to be understood as the natural scaling of the corresponding penalty terms. In particular, 
we observe a stronger penalization of the drift of sites away from the associated cell  
and the fusion of pairs of sites in case of larger cell masses.

It might happen that in the limit along a maximizing sequence 
$(\sitesVector^N,\gVector^N)_N\subset\mathcal{X}_l\times\R^l$
of $\mathcal{F}_l^\eta$ for some $l\in \N$ 
Laguerre cells collapse and the effective number of cells
decreases. The following existence theorem takes this into account.

\begin{thrm} \label{thm:maximizer}
Let us assume that $\Phi\in C(\domain)$ with maximum $\bar \Phi$ on $D$.
Then, for given number of sites $l\in\N$ there exists an $n\leq l$, 
such that a maximizer $(\sitesVector^*,\gVector^*)$ of $\mathcal{F}_n^\eta$ exists  
with  $m_i[\sitesVector^*,\gVector^*]>0$ for $i=1,\ldots,n$
and  $\mathcal{F}_n^\eta(\sitesVector^*,\gVector^*)\geq\sup_{(\sitesVector,\gVector)}\mathcal{F}_l^\eta(\sitesVector,\gVector)$.
\end{thrm}
\begin{proof}
Consider a maximizing sequence for $\mathcal{F}_l^\eta$. 
Since $m_i[\sitesVector^N,\gVector^N]\in[0,1],$ and $b_i[\sitesVector^N,\gVector^N]\in\domain$ for $m_i[\sitesVector^N,\gVector^N]>0$, we may assume that, up to the selection of a subsequence, the sequences of masses and barycenters converge to limits $m^*_i$ and $b^*_i$, respectively, and that there exists an $n\leq l$ with $m^*_i>0$ if and only if 
$i \leq n$. Then, we obtain
\begin{align*}
\sup_{(\sitesVector,\gVector)}\mathcal{F}_l^\eta(\sitesVector,\gVector)&=\lim_{N\rightarrow\infty}\mathcal{F}_l^\eta(\sitesVector^N,\gVector^N)\leq\limsup_{N\rightarrow\infty}\mathcal{F}_l(\sitesVector^N,\gVector^N)-\eta\liminf_{N\rightarrow\infty}\mathcal{R}_l(\sitesVector^N,\gVector^N)\\
&\leq\limsup_{N\rightarrow\infty}\mathcal{F}_n(\sitesVector^N,\gVector^N)+\limsup_{N\rightarrow\infty}\sum_{i=n+1,\ldots, l \atop m_i[\sitesVector^N,\gVector^N]>0} m_i[\sitesVector^N,\gVector^N]\Phi(b_i[\sitesVector^N,\gVector^N])-\eta\liminf_{N\rightarrow\infty}\mathcal{R}_n(\sitesVector^N,\gVector^N)\\
&\quad -\eta\liminf_{N\rightarrow\infty}
\left(\mathcal{R}_l(\sitesVector^N,\gVector^N)-\mathcal{R}_n(\sitesVector^N,\gVector^N)\right)\\
&=\lim_{N\rightarrow\infty}\mathcal{F}_n(\sitesVector^N,\gVector^N)+0-\eta\lim_{N\rightarrow\infty}\mathcal{R}_n(\sitesVector^N,\gVector^N)-\eta\liminf_{N\rightarrow\infty}\left(\mathcal{R}_l(\sitesVector^N,\gVector^N)-\mathcal{R}_n(\sitesVector^N,\gVector^N)\right)\\
&\leq\lim_{N\rightarrow\infty}\mathcal{F}_n(\sitesVector^N,\gVector^N)-\eta\lim_{N\rightarrow\infty}\mathcal{R}_n(\sitesVector^N,\gVector^N)=\lim_{N\rightarrow\infty}\mathcal{F}_n(\sitesVector^N,\gVector^N)-\eta\mathcal{R}_n(\sitesVector^N,\gVector^N)\\
&=\lim_{N\rightarrow\infty}\mathcal{F}_n^\eta(\sitesVector^N,\gVector^N).
\end{align*}
Hence, if it exists, a maximizer of $\mathcal{F}_n^\eta$ attains a greater or equal objective value than a maximizer of 
$\mathcal{F}_l^\eta$. 

Now, let $(\sitesVector^N,\gVector^N)_N\subset\sitesSet\times\R^n$ be a maximizing sequence for $\mathcal{F}_n^\eta$ and in analogy to before assume that $m_i[\sitesVector^N,\gVector^N]$ and $b_i[\sitesVector^N,\gVector^N]$ 
have limits $m^*_i>0$, and $b^*_i\in D$ for $i=1,\ldots, n$, that $\mathcal{F}_n^\eta(\sitesVector^N,\gVector^N)$ is monotonically increasing in $N$ and that $m_i[\sitesVector^N,\gVector^N]\geq\tfrac12 m_i^*$. Consequently, for $\bar \Phi$ being the maximal value of $\Phi$ on $D$ the estimate
\begin{align}
\mathcal{F}_n^\eta(\sitesVector^0,\gVector^0) &\leq \nu(D) \bar \Phi - 
\eta\left(\sum_{i=1,\ldots, n}\int_\domain\vert y-x_i^N\vert^2\chi_{\Laguerre_i[\sitesVector^N,\gVector^N]}\d\nu(y)+\sum_{1\leq i,j \leq n\atop i\neq j}\frac{m_i[\sitesVector^N,\gVector^N]m_j[\sitesVector^N,\gVector^N]}{\vert x_i^N-x_j^N\vert^2}\right) \nonumber\\
&\leq \bar \Phi -\eta\left(\frac12\sum_{i=1,\ldots,n}\dist^2(x_i^N,\domain)m_i^*+\frac14\sum_{1\leq i,j \leq n \atop i\neq j}\frac{m_i^*m_j^*}{\vert x_i^N-x_j^N\vert^2}\right)
\end{align}
is obtained. This implies the following a priori bounds:
\begin{align*}
\dist(x_i^N,\domain) \leq \sqrt{ \frac{2(\bar \Phi - \mathcal{F}_n^\eta(\sitesVector^0,\gVector^0))}{\eta m_i^*}} ,\quad
\vert x_i^N -x_j^N\vert \geq \sqrt{\frac{\eta m_i^*m_j^*}{4( \bar \Phi - \mathcal{F}_n^\eta(\sitesVector^0, \gVector^0))}}
\end{align*}
for all $N$. For the uniform bound on $\gVector$, recall that $\gVector$ and $\gVector+\lambda\Id_n$ both induce the same Laguerre diagram. Hence, we may assume without loss of generality that $\g_1^N=0$ for all $N\in\N$. 
Then,  $\lim_{N\rightarrow\infty}\g_j^N=\infty$ for $j=2,\ldots,n$ 
would imply $\lim_{N\rightarrow\infty} m_1[\sitesVector^N,\gVector^N]=0$, which contradicts our choice of $n$.
Similarly, $\lim_{N\rightarrow\infty}\g_j=-\infty$ would imply $\lim_{N\rightarrow\infty} m_j[\sitesVector^N,\gVector^N]=0$, 
which again is a contradiction.
Hence, $\vert\gVector^N\vert\leq C$ for some $C>0$ and all $N\in \N$. 
Finally, given these a priori bounds, the existence of a maximizer of $\mathcal{F}_n^\eta$ follows directly from the Weierstra{\ss} extreme value theorem. 
\end{proof}

\section{Semi-discrete optimal transport revisited}\label{sec:SD}
In the previous section we stated existence results of maximizers of the sender's revenue over the class of Laguerre partitions. 
To compute these optimal Laguerre partitions we recall in this section the connection of Laguerre diagrams 
to solutions of semi-discrete optimal transport problems and the associated dual formulation.

In optimal transport theory, one considers optimal couplings $\Pi\in\Prob(\domain \times \domain) \in U(\mu,\nu)$ of
two probability measures $\mu,\nu\in\Prob(\domain)$. Here,  $\mathcal{U}(\mu,\nu)$ is the set $\Pi\in\Prob(\domain \times \domain)$ with 
$\Pi(A\times\domain)=\mu(A)$ and $\Pi(\domain\times A)=\nu(A)$ for all Borel sets $A \subset \domain$. For the given 
cost function 
$(x,y)  \mapsto \vert x-y\vert^2$ on $\domain \times \domain$ measuring the cost of transport from $x$ to $y$ on $\domain$, a coupling $\Pi$ is optimal if it 
minimizes 
\begin{align}\label{kantorovich_primal_sd}
\Wass^2[\mu,\nu]\coloneqq \inf_{\Pi\in \mathcal{U}(\mu,\nu)}\int_{\domain \times \domain} \vert x-y\vert^2 \d\Pi(x,y).
\end{align}
The functional $\Wass(\mu,\nu)$ is called the 2-Wasserstein distance and defines a metric between the probability measures $\mu$ and $\nu$ in $\Prob(\domain)$ (cf. \cite[Chapter 5]{Sa15}).
In semi-discrete optimal transport the measure $\mu$ is assumed to be a discrete (empirical) probability measure, i.e.
$$\mu\coloneqq\sum_{i=1,\ldots, n} m_i \delta_{x_i}$$
with $\sum_{i=1,\ldots n} m_i  = 1$ and $m_i \geq 0$.
The minimization of \eqref{kantorovich_primal_sd} is a constrained linear minimization problem and denoted the primal Kantorovich problem.
As such, it can naturally be paired with a constrained linear maximization problem as the dual problem, the dual Kantorovich problem \cite[Section 6.1]{AmGi13}.
We obtain
\begin{align}\label{kantorovich_dual_sd} 
\Wass^2[\mu,\nu]=\sup_{(f,\gVector)\in\mathcal{R}}\discreteMeasure\cdot\gVector+\int_\domain f(y)\d\nu(y),
\end{align}
where $\mathbf m\coloneqq(m_1,\ldots,m_n)$, $\gVector\coloneqq(\g_1,\ldots,\g_n)$, and 
$$\mathcal{R}\coloneqq\lbrace (f,\gVector)\in C(\domain)\times\R^n: 
f(y)+\g_i\leq\vert x_i-y\vert^2 \text{  for all  } i=1,\ldots,n \rbrace.$$ 
For given $\gVector\in \R^n$ we obtain for the optimal $f$ which is consistent with the constrained condition $(f,\gVector)\in\mathcal{R}$ that
$f(y) = \gVector^{C}(y)$ for all $y\in D$ with  
\begin{align}
\gVector^{C}(y)\coloneqq\min_{1\leq i \leq n}\vert y-x_i\vert^2 - \text{g}_i.
\end{align}
Given the $C$-transform $\gVector^{C}: \domain \to \R^D$  
one can reformulate the dual Kantorovich problem \eqref{kantorovich_dual_sd} 
as the unconstrained convex program 
\begin{align}\label{dual}
\Wass^2[\mu,\nu]= \max_{\gVector\in\R^n} \Dual[\gVector]
\end{align}
\begin{align}\label{SDOT_prob}
\Dual[\gVector] &\coloneqq \int_\domain\gVector^C(y)d\nu(y) +\discreteMeasure\cdot\gVector
=\sum_{i=1,\ldots,n} \int_{\Laguerre_i[\sitesVector,\gVector]}(\vert y-x_i\vert^2-\g_i)d\nu(y)
+\discreteMeasure\cdot\gVector,
\end{align}
where $\Laguerre_i[\sitesVector,\gVector]$ are the Laguerre cells associated with the weight vector $\gVector$ and the vector of fixed sites $\sitesVector$ defined in \eqref{eq:LaguerreCell}.
Hence, solving \eqref{SDOT_prob} for $\nu$ and $\discreteMeasure$ consists in finding a Laguerre cell partition of $\domain$ described via the weight
vector $\gVector\in\R^n$ with cells centered at the given $\sitesVector$, and  with $\discreteMeasure =(m_i)_{i=0,\ldots, n}$ and $m_i=m(\Laguerre_i[\sitesVector,\gVector])$ for all $i\in \{1,\ldots,n\}$. 
In what follows, we will need to differentiate the functions  
$\mathcal{F}_n^\eta$ 
defined in \eqref{eq:objective} with respect to the weights $g_j$.
For the differentiation of $\Dual(\gVector)$ we obtain
\begin{align}\label{partialDual}
\partial_{g_j} \Dual[\gVector] =  - \int_{\domain} \chi_{\Laguerre_j[\sitesVector,\gVector]}\d\nu(y) + m_j.
\end{align}
\section{Entropy regularization}\label{sec:entropy}
Let us recall that our goal is to maximize the function $\mathcal{F}_n^\eta$ via an optimization of the 
Laguerre cells $\Laguerre_j[\sitesVector,\gVector]$ 
described in terms of the sites $x_j$ and the weights $\g_j$ for $j=1,\ldots, n$. 
In general, changing the sites and the weights fosters topological changes in the diagram's topology, leading 
to a challenging combinatorial optimization problem. 
To avoid this,  we introduce in this section the entropy relaxed optimal transport formulation \cite{PeCu19} of the associated 
semi-discrete optimal transport described in Section~\ref{sec:Opti}, 
which will allow us to derive a computationally efficient algorithm to optimize over 
Laguerre partitions and thus solve the moment Bayesian persuasion problem numerically. 
An extensive overview on entropy regularization of optimal transport is also given by Chewi et. al. in \cite{ChNi24}.

We begin by considering the regularized Wasserstein distance 
\begin{align}\label{kantorovich_primal_reg}
\Wass^2_\varepsilon[\mu,\nu]\coloneqq\inf_{\Pi\in \mathcal{U}(\mu,\nu)}\int_{\{x_1,\ldots,x_n\}\times\domain} c(x,y)\d\Pi(x,y)+\varepsilon\KL[\Pi\vert\xi]
\end{align}
with transport cost $c(x,y)= \vert x-y\vert^2$ from $x$ to $y$ and 
for a regularization parameter $\varepsilon>0$. We consider a measure $\xi\in\Prob(\{x_1,\ldots,x_n\}\times\domain)$ with 
$\supp \mu \otimes \nu \subseteq \supp \xi$ for $\mu$ and $\nu$ as above, and we define
the Kullback-Leibler divergence 
\begin{align}
\KL[\Pi\vert\xi]\coloneqq\int_{\{x_1,\ldots,x_n\}\times\domain}\log\left(\frac{\d\Pi}{\d\xi}(x,y)\right)\d\Pi(x,y)+(\d\xi(x,y)-\d\Pi(x,y))
\end{align}
between the measures $\Pi$ and $\xi$ 
in the case that $\Pi$ is absolutely continuous with respect to $\xi$, and otherwise $\KL[\Pi\vert\xi]\coloneqq\infty$. 
The Kullback-Leibler divergence is a concave functional measuring the dissimilarity of the measures $\Pi$ and $\xi$ and acts here as a regularizing entropy functional.
The standard choice for $\xi$ is $\xi = \mu \otimes \nu$. In fact, as long as the above support property holds 
the functional in  \eqref{kantorovich_primal_reg} only changes by an additive constant and hence, the minimizer 
remains the same. To simplify the optimization algorithm for an entropy regularization of the functional 
$\mathcal{F}_l^\eta$ defined in  \eqref{eq:objective3} $\xi = \Count \otimes \Leb$ is a particularly suitable choice where $\Count = \sum_{j=1,\ldots, n} \delta_{x_j}$ is the counting measure on the support of $\mu$ 
and $\Leb$ the Lebesque measure. For this choice we proceed as follows.

Associated with the constraint optimization problem \eqref{kantorovich_primal_reg} 
is the Lagrangian 
\begin{align*}
\Lagrangian^\varepsilon(\Pi,f,g)= & \int_{\domain \times \domain}c(x,y)\d\Pi(x,y)+\varepsilon\int_{\domain \times \domain}\left(\log\frac{\d\Pi}{\d\xi}-1\right)(x,y)\d\Pi(x,y)\\
&-\left(\int_{\domain \times \domain}f(y)\d\Pi(x,y)-\int_\domain f(y)\d\nu(y)\right)-\left(\int_{\domain \times \domain}g(x)\d\Pi(x,y)-\int_\domain g(x)\d\mu(x)\right).
\end{align*}
Now, assume that $\Pi$ has the density $p$ with respect to $\xi$, i.e.~$\d\Pi(x,y)=p(x,y)\d\xi(x,y)$.
Then, a necessary condition of a saddle point is that the derivative of $\Lagrangian^\varepsilon$
vanishes in all directions $q: (x,y) \mapsto q(x,y)$. Hence, 
\begin{align} \nonumber
0 =\left(\frac{\partial\Lagrangian^\varepsilon}{\partial p}(p)\right)(q)
=&\int_{\domain \times \domain}c(x,y)q(x,y)\d\xi(x,y)+
\varepsilon\int_{\domain \times \domain}\left(\log p(x,y)-1\right)q(x,y)\d\xi(x,y)\\
&+\varepsilon\int_{\domain \times \domain}q(x,y)\d\xi(x,y)
-\int_{\domain \times \domain}f(y)q(x,y)\d\xi(x,y)-\int_{\domain \times \domain}g(x)q(x,y)\d\xi(x,y).
\end{align}
Hence, we obtain $c(x,y)+\varepsilon\log(p(x,y))-f(y)-g(x)=0$ pointwise, or equivalently 
$$p(x,y)=\frac{\d\Pi}{\d\xi}(x,y)=\exp\left(\frac{-c(x,y)+g(x)+f(y)}{\varepsilon}\right).$$
and 
\begin{align*}
\Lagrangian^\varepsilon(\Pi(f,g),f,g)&=\int_\domain f(y)\d\nu(y)+\int_\domain g(x)\d\mu(x)-\varepsilon\int_{\domain \times \domain}\exp\left(\frac{-c(x,y)+g(x)+f(y)}{\varepsilon}\right)\d\xi(x,y) \\
&= \int_\domain f(y)\d\nu(y)+\discreteMeasure\cdot\gVector 
-\varepsilon\int_{\domain} \sum_{j=1,\ldots, n}  
\exp\left(\frac{-c(x_j,y)+g(x_j)+f(y)}{\varepsilon}\right)\d\Leb(y)
\end{align*}
Finally, we obtain a dual, entropy regularized, unconstrained formulation for the Wasserstein functional 
\begin{align} \nonumber
\Wass^2_\varepsilon[\mu,\nu]
&\coloneqq \sup_{(f \in C(\domain) ,\gVector \in \R^n)} \Lagrangian^\varepsilon(\Pi(f,g),f,g) \\
\label{kantorovich_dual_reg}
&=\sup_{(f \in C(\domain) ,\gVector \in \R^n)} \left(
\discreteMeasure\cdot\gVector+\int_\domain f(y)\d\nu(y) -  \varepsilon \int_\domain   \sum_{j=1,\ldots, n} 
\exp\left(\frac{-\vert y-x_j\vert^2+ f(y) + \g_j}{\varepsilon}\right) \d \Leb(y) \right)
\end{align}
Here, the exponent $\exp(\varepsilon^{-1}(f(y) + \g_j-\vert y-x_j\vert^2))$ acts as a soft penalty in place of the original hard constraint 
$f(y) + \g_j \leq \vert y-x_j\vert^2$.
The optimal $f$ for fixed $g$ is characterized by 
\begin{align*}
0&=\partial_f\Lagrangian^\varepsilon(\Pi(f,g),f,g)(r)
=\int_\domain r(y)\d\nu(y)-\int_{\domain \times \domain}\exp\left(\frac{-c(x,y)+g(x)+f(y)}{\varepsilon}\right)r(y)\d\xi(x,y)\\
&=\int_\domain r(y)\left(\frac{\d\nu}{\d\Leb}\right)(y)\d\Leb(y)-\int_{\domain \times \domain}\exp\left(\frac{-c(x,y)+g(x)+f(y)}{\varepsilon}\right)r(y)\d\Count(x)\otimes\d\Leb(y)
\end{align*}
for all directions $r$.
Hence, we obtain $$0=\left(\frac{\d\nu}{\d\Leb}\right)(y)-\int_\domain\exp\left(\frac{-c(x,y)+g(x)+f(y)}{\varepsilon}\right)\d\Count(x),$$
and finally
\begin{align}\label{eq:dualeps}
\gVector^{C,\varepsilon}(y)\coloneqq\varepsilon\log\left(\frac{\d\nu}{\d\Leb}\right)(y) -\varepsilon\log\left(\sum_{j=1,\ldots,n}\exp\left(\frac{\g_j-\vert y-x_j\vert^2}{\varepsilon}\right)\right)
\end{align}
defines the optimal $f$ for given $\gVector\in \R^n$ and all $y\in \domain$.
Given this $C$-transform, the regularized dual formulation in \eqref{kantorovich_dual_reg} can be rewritten as
$\Wass^2_\varepsilon[\nu,\mu] = \max_{\gVector\in\R^n}\Dual^\varepsilon(\gVector)$
with 
\begin{align}\label{eq:sdot_reg_pb}
\Dual^\varepsilon[\gVector] 
&\coloneqq \Lagrangian^\varepsilon(\Pi(\gVector^{C,\varepsilon},g),\gVector^{C,\varepsilon},g) \\
&=\int_\domain \gVector^{C,\varepsilon}(y)\d\nu(y) +\gVector\cdot\discreteMeasure 
- \varepsilon \int_\domain \left(\frac{\d\nu}{\d\Leb}\right)(y) \d\Leb(y) 
= \int_\domain \gVector^{C,\varepsilon}(y)\d\nu(y) +\gVector\cdot\discreteMeasure - \varepsilon
\end{align}
A sufficient condition for a vector $\gVector\in\R^n$ to maximize \eqref{eq:sdot_reg_pb} for given sites $\sitesVector$ and masses $\discreteMeasure$ is 
\begin{align}
0=\partial_{\g_i} \Dual^\varepsilon[\gVector]  = \partial_{\g_j}\left(\int_\domain \gVector^{C,\varepsilon}(y)d\nu(y) +\gVector\cdot\discreteMeasure\right) &=
\int_\domain \partial_{\g_i} \gVector^{C,\varepsilon}(y)\d\nu(y)+m_i\\
&\text{ with }\partial_{\g_i} \gVector^{C,\varepsilon}(y) = 
- \frac{\exp\left(\frac{\g_i-\vert y-x_i\vert^2}{\varepsilon}\right)}{\sum_{j=1,\ldots,n}
\exp\left(\frac{\g_j-\vert y-x_j\vert^2}{\varepsilon}\right)}
\end{align}
for $i=1,\ldots, n$.
The set of functions $\{\chi^\epsilon_i[\sitesVector,\gVector]\}_{i=1,\ldots, n}$
with $\chi^\epsilon_i[\sitesVector,\gVector](y)= - \partial_{\g_i} \gVector^{C,\varepsilon}(y)$ for $y\in \domain$
forms a partition of unity on $\domain$. 
Furthermore,
\begin{align}
\quad \lim_{\varepsilon\rightarrow 0}
\chi_i^\varepsilon[\sitesVector,\gVector] =\chi_{\Laguerre_i[\sitesVector,\gVector]}
\end{align} 
for $i=1,\ldots,n$, where the convergence is in  $L^1(\nu)$. 
For given $[\sitesVector,\gVector]$  with 
$\varepsilon>0$, $\sitesVector=(x_1,\ldots, x_n) \in \domain^n$, $x_j\neq x_i$ for $j,\;i =1,\ldots, n$ 
\begin{align}
m_i^\varepsilon[\sitesVector,\gVector]\coloneqq\int_\domain\chi_i^\varepsilon[\sitesVector,\gVector](y)\d\nu(y), \quad 
b_i^\varepsilon[\sitesVector,\gVector]\coloneqq\int_\domain \frac{y\; \chi_i^\varepsilon[\sitesVector,\gVector](y)\d\nu(y)}{m_i^\varepsilon[\sitesVector,\gVector]}
\end{align}
define regularized masses and regularized barycenters, respectively. By the $L^1$-convergence of $\chi_i^\varepsilon[\sitesVector,\gVector]$ to $\chi_{\Laguerre_i}$ and by the compactness of $\domain$, one gets
\begin{align}\label{eq:convergenceeps}
\lim_{\varepsilon\rightarrow 0} m_i^\varepsilon[\sitesVector,\gVector] =m_i[\sitesVector,\gVector], \quad
\lim_{\varepsilon\rightarrow 0} b_i^\varepsilon[\sitesVector,\gVector]=b_i[\sitesVector,\gVector] \text{ for }  m_i[\sitesVector,\gVector] >0.
\end{align}
Consequently, we obtain the entropy-regularized cost functional 
\begin{align}\label{infosec_pob_reg}
\mathcal{F}_n^\varepsilon[\sitesVector,\gVector]\coloneqq
\sum_{i=1,\ldots, n} 
m_i^\varepsilon[\sitesVector,\gVector]\Phi(b_i^\varepsilon[\sitesVector,\gVector])
\end{align}
as an approximation of the original cost functional $\mathcal{F}_n$ on power diagrams defined in \eqref{eq:objective}. Analogously, one finally obtains the entropy-regularized cost functional with penalty parameter $\eta>0$:
\begin{align}\label{infosec_pob_reg_penalty}
\mathcal{F}_n^{\eta,\varepsilon}[\sitesVector,\gVector]\coloneqq
\mathcal{F}_n^{\varepsilon}[\sitesVector,\gVector] - \eta\mathcal{R}_n^\varepsilon[\sitesVector,\gVector],
\end{align}
where the regularized penalty term is defined as
\begin{align}\label{def:reg_eps}
\mathcal{R}_n^\varepsilon[\sitesVector,\gVector]\coloneqq\sum_{i=1,\ldots, n}\int_\domain\vert y-x_i\vert^2\chi_i^\varepsilon[\sitesVector,\gVector](y)\d\nu(y)+\sum_{1\leq i,j \leq n \atop {i\neq j}}\frac{m_i^\varepsilon[\sitesVector,\gVector]m_j^\varepsilon[\sitesVector,\gVector]}{\vert x_i-x_j\vert^2}.
\end{align}
in analogy to \eqref{def:regularizer}.

{
We are now in the position to prove the convergence of maximizers of the entropy regularized functional $\mathcal{F}_n^{\eta,\varepsilon}$ 
given in  \eqref{infosec_pob_reg_penalty} to a
maximizer of the original functional defined in \eqref{eq:objective3} for $\varepsilon\to 0$.
\begin{prpstn}\label{prop:gamma_eps}
Let $(\varepsilon_N)_N\subset\R^+$, $(\sitesVector^N,\gVector^N)_N\subset\sitesSet\times\R^n$ so that $\varepsilon_N\rightarrow 0$ and $(\sitesVector^N,\gVector^N)\rightarrow(\sitesVector,\gVector)\in\sitesSet\times\R^n$ for $N\rightarrow\infty$. Then, we have that 
\begin{align*}
\lim_{N\rightarrow\infty}\chi_i^{\varepsilon^N}[\sitesVector^N,\gVector^N]=\chi_{\Laguerre_i[\sitesVector,\gVector]},
\end{align*}
where the limit is taken in $L^1(\nu).$
\end{prpstn}
\begin{proof}
Let $y\in D$ be in the interior of $\Laguerre_i(\sitesVector,\gVector).$ Then, there is $\delta>0$, such that $\vert y-x_i\vert^2-\g_i\leq\vert y-x_j\vert^2-\g_j-3\delta,$ for all $j\neq i$. and by the continuity of the function 
$(\sitesVector,\gVector)\mapsto\vert y-x_i\vert^2-\g_i$ we obtain that $\vert y-x_i^N\vert^2-\g_i^N\leq\vert y-x_j^N\vert^2-\g_j^N-\delta$ for sufficiently large $N$. 
This implies that
\begin{align*}
\chi_i^\varepsilon[\sitesVector^N,\gVector^N](y)=&\frac{\exp\left(\frac{\g_i^N-\vert y-x_i^N\vert^2}{\varepsilon^N}\right)}{\sum_{j=1}^n
\exp\left(\frac{\g_j^N-\vert y-x_j^N\vert^2}{\varepsilon^N}\right)}\geq\frac{\exp\left(\frac{\g_i^N-\vert y-x_i^N\vert^2}{\varepsilon^N}\right)}{\exp\left(\frac{\g_i^N-\vert y-x_i^N\vert^2}{\varepsilon^N}\right)+(n-1)
\exp\left(\frac{\g_i^N-\vert y-x_i^N\vert^2}{\varepsilon^N}\right)\exp(-\delta/\varepsilon^N)}\\
=&\frac{1}{1+(n-1)\exp(-\delta/\varepsilon^N)}\rightarrow 1
\end{align*}
for $N\rightarrow\infty$. Now, let $y\notin\overline{\Laguerre_i(\sitesVector,\gVector)}.$ Then, there is $j\neq i$, and $\delta>0$, such that $\vert y-x_i\vert^2-\g_i\geq\vert y-x_j\vert^2-\g_j+3\delta.$ and 
$\vert y-x_i^N\vert^2-\g_i^N\geq\vert y-x_j^N\vert^2-\g_j^N+\delta$ for sufficiently large $N$. From this it follows that
\begin{align*}
\chi_i^\varepsilon[\sitesVector^N,\gVector^N](y)=&\frac{\exp\left(\frac{\g_i^N-\vert y-x_i^N\vert^2}{\varepsilon^N}\right)}{\sum_{j=1}^n
\exp\left(\frac{\g_j^N-\vert y-x_j^N\vert^2}{\varepsilon^N}\right)}\leq\frac{\exp\left(\frac{\g_i^N-\vert y-x_i^N\vert^2}{\varepsilon^N}\right)}{\exp\left(\frac{\g_j^N-\vert y-x_j^N\vert^2}{\varepsilon^N}\right)}=\exp\left(\frac{\g_i^N-\g_j^N-\vert y-x_i^N\vert^2+\vert y-x_j^N\vert^2}{\varepsilon^N}\right)\\
\leq&\exp\left(-\frac{\delta}{\varepsilon^N}\right)\rightarrow 0,
\end{align*}
as $N\rightarrow\infty$. By the absolute continuity of $\nu$ with respect to the Lebesgue measure, we have proven that $\chi_i^{\varepsilon^N}[\sitesVector^N,\gVector^N]\rightarrow\chi_{\Laguerre_i[\sitesVector,\gVector]}$, $\nu$-almost everywhere, and by the dominated convergence theorem, we obtain the claim.
\end{proof}
In preparation for the convergence of maximizers, we now state the entropy-regularized "pendant" of Theorem ~\ref{thm:maximizer}.
\begin{thrm} 
Let us assume that $\varepsilon>0$, and $\Phi\in C(\domain)$ with maximum $\bar \Phi$ on $D$.
Then, for given number of sites $l\in\N$ there exists an $n\leq l$, 
such that a maximizer $(\sitesVector^*,\gVector^*)$ of $\mathcal{F}_n^{\eta,\varepsilon}$ exists  
with  $m_i^\varepsilon[\sitesVector^*,\gVector^*]>0$ for $i=1,\ldots,n$
and  $\mathcal{F}_n^{\eta,\varepsilon}(\sitesVector^*,\gVector^*)\geq\sup_{(\sitesVector,\gVector)}\mathcal{F}_l^{\eta,\varepsilon}(\sitesVector,\gVector)$.
\end{thrm}
\begin{proof}
The proof is analogous to Theorem \ref{thm:maximizer}.
\end{proof}

We are now in the position to prove the convergence of maximizers of the entropy regularized functional $\mathcal{F}_n^{\eta,\varepsilon}$ 
given in  \eqref{infosec_pob_reg_penalty} to a
maximizer of the original functional $\mathcal{F}_n^\eta$ defined in \eqref{eq:objective3} for $\epsilon\to 0$.
\begin{thrm}\label{thm:max_conv_eps}
Let $\Phi\in C(\domain),$ $l\in\N$ and $\eta>0$. For each $N\in\N$, consider maximizers $(\sitesVector^N,\gVector^N)\in\mathcal{X}_{\beta(N)}\times\R^{\beta(N)}$, where $\beta:\N\rightarrow\{1,\ldots,n\}$, so that $\Func_{\beta(N)}^{\eta,\varepsilon_N}(\sitesVector^N,\gVector^N)\geq\max_{k=1,\ldots,n}\{\sup_{(\sitesVector,\gVector)\in\mathcal{X}_k\times\R^k}\Func_k^{\eta,\varepsilon_N}(\sitesVector,\gVector)\}$.
Then, there is $n\in\{1,\ldots,l\}$ such that up to the selection of a subsequence $(\sitesVector^{N},\gVector^{N})_{N\in\N}\subset\sitesSet\times\R^n$ converges to a limit $(\sitesVector^*,\gVector^*)\in\sitesSet\times\R^{n}$ for which it holds $$\Func_{n}^{\eta}(\sitesVector^*,\gVector^*)\geq\max_{k=1,\ldots,l}\left\{\sup_{(\sitesVector,\gVector)\in\mathcal{X}_k\times\R^k}\Func_k^{\eta}(\sitesVector,\gVector)\right\}.$$
\end{thrm}
\begin{proof}
The proof of the first statement follows the proof of Theorem \ref{thm:maximizer} very closely, albeit with a few key differences, which will be highlighted now: Let $n'\in\{1,\ldots,l\}$, such that $\vert\{N:\beta(N)=n'\}\vert=\infty$. Choose a subsequence such that up to relabeling of indices $(\sitesVector^N,\gVector^N)_{N\in\N}\subset\mathcal{X}_{n'}\times\R^{n'}$, such that the sequence of masses $(m_1^{\varepsilon_N}[\sitesVector^N,\gVector^N],\ldots,m_{n'}^{\varepsilon_N}[\sitesVector^N,\gVector^N])_{N\in\N}$ converges to some $(m_1^*,\ldots,m_{n'}^*)\in[0,1]^{n'}$. Moreover, assume without loss of generality that there is $n\in\{1,\ldots,n'\}$ with $m_i^*>0$ if and only if $i\leq n$, and $m_i^{\varepsilon_N}[\sitesVector^N,\gVector^N]\geq\tfrac12m_i^*$ for all $N\in\N$, $i\leq n$. It then holds for all $k \in\{1,\ldots,l\}$ and all $(\overline\sitesVector,\overline\gVector)\in\mathcal{X}_k\times\R^k$ with $m_i[\bar\sitesVector,\bar\gVector]>0$ for $i=1,\ldots,k$:
\begin{align*}\Func_k^\eta[\overline\sitesVector,\overline\gVector]&=\lim_{N\rightarrow\infty}\Func_k^{\eta,\varepsilon_N}[\overline\sitesVector,\overline\gVector]\leq\limsup_{N\rightarrow\infty}\Func_{n'}^{\eta,\varepsilon_N}[\sitesVector^N,\gVector^N]\leq\limsup_{N\rightarrow\infty}\mathcal{F}_{n'}^{\varepsilon_N}(\sitesVector^N,\gVector^N)-\eta\liminf_{N\rightarrow\infty}\mathcal{R}_{n'}^{\varepsilon_N}(\sitesVector^N,\gVector^N)\\
&\leq\limsup_{N\rightarrow\infty}\mathcal{F}_n^{\varepsilon_N}(\sitesVector^N,\gVector^N)+\limsup_{N\rightarrow\infty}\sum_{i=n+1,\ldots, n'} m_i^{\varepsilon_N}[\sitesVector^N,\gVector^N]\Phi(b_i^{\varepsilon_N}[\sitesVector^N,\gVector^N])-\eta\liminf_{N\rightarrow\infty}\mathcal{R}_n^{\varepsilon_N}(\sitesVector^N,\gVector^N)\\
&\quad -\eta\liminf_{N\rightarrow\infty}
\left(\mathcal{R}_{n'}^{\varepsilon_N}(\sitesVector^N,\gVector^N)-\mathcal{R}_n^{\varepsilon_N}(\sitesVector^N,\gVector^N)\right)\\
&=\limsup_{N\rightarrow\infty}\mathcal{F}_n^{\varepsilon_N}(\sitesVector^N,\gVector^N)+0-\eta\liminf_{N\rightarrow\infty}\mathcal{R}_n^{\varepsilon_N}(\sitesVector^N,\gVector^N)-\eta\liminf_{N\rightarrow\infty}\left(\mathcal{R}_{n'}^{\varepsilon_N}(\sitesVector^N,\gVector^N)-\mathcal{R}_n^{\varepsilon_N}(\sitesVector^N,\gVector^N)\right)\\
&\leq\limsup_{N\rightarrow\infty}\mathcal{F}_n^{\varepsilon_N}(\sitesVector^N,\gVector^N)-\eta\liminf_{N\rightarrow\infty}\mathcal{R}_n^{\varepsilon_N}(\sitesVector^N,\gVector^N)\leq\bar\Phi-\eta\liminf_{N\rightarrow\infty}\mathcal{R}_n^{\varepsilon_N}(\sitesVector^N,\gVector^N).
\end{align*}
We obtain $\liminf_{N\rightarrow\infty}\mathcal{R}_n^{\varepsilon_N}(\sitesVector^N,\gVector^N)\leq\frac{\bar\Phi-\Func_k^\eta[\overline\sitesVector,\overline\gVector]}{\eta}$. Without loss of generality, we can assume that for all $\delta>0$, $\mathcal{R}_n^{\varepsilon_N}(\sitesVector^N,\gVector^N)<\frac{\bar\Phi-\Func_k^\eta[\overline\sitesVector,\overline\gVector]+\delta}{\eta}$ holds for all $N\in\N$, $i=1,\ldots,l$. Then, we obtain the bounds
\begin{align*}
\dist(x_i^N,\domain) < \sqrt{ \frac{2(\bar \Phi - \mathcal{F}_k^\eta[\bar\sitesVector,\bar\gVector]+\delta)}{\eta m_i^*}} ,\quad
\vert x_i^N -x_j^N\vert > \sqrt{\frac{\eta m_i^*m_j^*}{4( \bar \Phi - \mathcal{F}_k^\eta[\bar\sitesVector, \bar\gVector]+\delta)}},
\end{align*}
for all pairwise different $i,j=1,\ldots,n$. Using the same arguments as in the proof of Theorem \ref{thm:maximizer}, one also finds analogous uniform bounds on $\vert\g^N\vert$. Hence, up to a subsequence, $(\sitesVector^N,\gVector^N)_{N\in\N}$ converges to some $(\sitesVector^*,\gVector^*)\in\sitesSet\times\R^n$. By the above computations, we finally obtain
\begin{align*}
\Func_k^\eta[\overline\sitesVector,\overline\gVector]\leq\lim_{N\rightarrow\infty}\mathcal{F}_n^{\varepsilon_N}(\sitesVector^N,\gVector^N)-\eta\mathcal{R}_n^{\varepsilon_N}(\sitesVector^N,\gVector^N)=\lim_{N\rightarrow\infty}\mathcal{F}_n^{\eta,\varepsilon_N}(\sitesVector^N,\gVector^N)=\mathcal{F}_n^{\eta}(\sitesVector^*,\gVector^*).
\end{align*}
For the last equality, we used the fact that the masses $m_i^{\varepsilon_N}$ do not vanish in the limit and Proposition \ref{prop:gamma_eps} for the convergence of masses $m^{\varepsilon_N}_i[\sitesVector^N,\gVector^N]\to m_i[\sitesVector^*,\gVector^*]$. For the convergence of barycenters $b^{\varepsilon_N}_i[\sitesVector^N,\gVector^N]\to b_i[\sitesVector^*,\gVector^*]$, we additionally use the fact that the function $y\mapsto y$ is bounded in $\domain$. Finally, for the convergence of $$\int_\domain\vert y-x_i^N\vert^2\chi_i^{\varepsilon_N}[\sitesVector^N,\gVector^N](y)\d\nu(y)\to\int_\domain\vert y-x_i^*\vert^2\chi_i[\sitesVector^*,\gVector^*](y)\d\nu(y),$$
we also use the fact that the family of functions $(y\mapsto\vert y-x_i^N\vert^2)_{N\in\N}$ is uniformly bounded in $L^\infty(\nu)$ and converges in $L^\infty(\nu)$ to $y\mapsto\vert y-x_i^*\vert^2$, and the uniform boundedness principle.
\end{proof}
To numerically implement the maximization of \ref{infosec_pob_reg} or \ref{infosec_pob_reg_penalty} via a gradient ascent approach we 
need to compute the gradient of $\mathcal{F}_n^\varepsilon$. Using $z=x_k$ or $z=\g_k$ for $k=1,\ldots, n$ we obtain
\begin{align*}
\partial_{z}\mathcal{F}_n^\varepsilon[\sitesVector,\gVector] &= 
\sum_{j=1,\ldots, n} 
\partial_z m_j^\varepsilon[\sitesVector,\gVector] \Phi(b_j^\varepsilon[\sitesVector,\gVector])
+  m_j^\varepsilon[\sitesVector,\gVector] \partial_z \Phi(b_j^\varepsilon[\sitesVector,\gVector]),\\
\partial_{x_k}\mathcal{R}_n^\varepsilon[\sitesVector,\gVector] &=
\int_\domain 2(x_k-y)\chi_k^\varepsilon[\sitesVector,\gVector](y)\d\nu(y)+\sum_{i=1,\ldots, n}\int_\domain\vert y-x_i\vert^2\partial_{x_k}\chi_i^\varepsilon[\sitesVector,\gVector](y)\d\nu(y)\\
&+\sum_{1\leq i,j \leq n \atop {i\neq j}}\left(\frac{\partial_{x_k}m_i^\varepsilon[\sitesVector,\gVector]m_j^\varepsilon[\sitesVector,\gVector]}{\vert x_i-x_j\vert^2}+\frac{m_i^\varepsilon[\sitesVector,\gVector]\partial_{x_k}m_j^\varepsilon[\sitesVector,\gVector]}{\vert x_i-x_j\vert^2}\right.\\
&-\left.2\frac{m_i^\varepsilon[\sitesVector,\gVector]m_j^\varepsilon[\sitesVector,\gVector]}{\vert x_i-x_j\vert^4}((x_i-x_j)\delta_{ik}+(x_j-x_i)\delta_{jk})\right) ,\\
\partial_{\g_k}\mathcal{R}_n^\varepsilon[\sitesVector,\gVector] &=
\sum_{i=1,\ldots, n}\int_\domain\vert y-x_i\vert^2\partial_{\g_k}\chi_i^\varepsilon[\sitesVector,\gVector](y)\d\nu(y)\\
&+\sum_{1\leq i,j \leq n \atop {i\neq j}}\left(\frac{\partial_{\g_k}m_i^\varepsilon[\sitesVector,\gVector]m_j^\varepsilon[\sitesVector,\gVector]}{\vert x_i-x_j\vert^2}+\frac{m_i^\varepsilon[\sitesVector,\gVector]\partial_{\g_k}m_j^\varepsilon[\sitesVector,\gVector]}{\vert x_i-x_j\vert^2}\right) ,\\
\partial_z \Phi(b_j^\varepsilon[\sitesVector,\gVector]) &= \nabla \Phi(b_j^\varepsilon[\sitesVector,\gVector]) 
\cdot \partial_z b_j^\varepsilon[\sitesVector,\gVector], \\
\partial_z m_j^\varepsilon[\sitesVector,\gVector] &= 
\int_\domain \partial_z  \chi_j^\varepsilon[\sitesVector,\gVector](y) \d\nu(y), \\
\partial_z b_j^\varepsilon[\sitesVector,\gVector] &=
\frac{1}{m_j^\varepsilon[\sitesVector,\gVector]} 
\int_\domain y \partial_z  \chi_j^\varepsilon[\sitesVector,\gVector](y) \d\nu(y)  
- \frac{b_j^\varepsilon[\sitesVector,\gVector]  \partial_z m_j^\varepsilon[\sitesVector,\gVector]}{m_j^\varepsilon[\sitesVector,\gVector]^2},\\
\partial_{x_k} \chi_j^\varepsilon[\sitesVector,\gVector](y) &=
\tfrac{2 x_j}{\varepsilon} \left(\chi_j^\varepsilon[\sitesVector,\gVector](y) -\delta_{kj}\right)  \chi_k^\varepsilon[\sitesVector,\gVector](y), \\
\partial_{\g_k} \chi_j^\varepsilon[\sitesVector,\gVector](y) &=
-\tfrac{1}{\varepsilon} \left(\chi_j^\varepsilon[\sitesVector,\gVector](y) -\delta_{kj}\right)  \chi_k^\varepsilon[\sitesVector,\gVector](y).
\end{align*}
Hence, since the gradient of $\powerDiagC^\varepsilon$ has the form of an integral with respect to $\nu$, a stochastic gradient ascent is a very suitable method to numerically obtain an optimizer of this problem: since evaluating the integral of gradients becomes very expensive, one can save computational costs at every iteration  by sampling the measure $\nu$, and by approximating the integral by the sum of the integrand evaluated at the samples' values, cf. \cite{Bo07}. This method has the advantage of not having to discretize the measure $\nu$ in space, which might lead to further inaccuracies.
Alternatively, to implement the maximization of \ref{infosec_pob_reg} numerically, we have to further discretize the continuous given measure $\nu$ in space.
\section{Spatial discretization}\label{sec:space}
As the next step, we have to discretize our optimization problem in space. To this end,
we restrict to the domain $D=[0,1]^d$ and consider a dyadic mesh with grid size $h=2^{-N}$ for $N\in \N$.
The domain is subdivided into cells 
$D^\alpha_h = \large\times_{1,\ldots, d} [\alpha_j h, (\alpha_j+1)h]$
where $\alpha$ is a multi-index in $(\mathcal{I}^h)^d$ for the index set 
$\mathcal{I}^h = \{0, \ldots, 2^N-1\}$.
We consider a discrete measure
$\nu^h = \sum_{\alpha\in \mathcal{I}^d} \nu_\alpha \delta_{y^\alpha_h}$
with $\nu_\alpha = \nu(D^\alpha_h)$, 
where  $y^\alpha_h = \left((\alpha_j+\tfrac12)h\right)_{j=1,\ldots, d}$
are the cell centers.
This discretization ansatz gives rise to discrete counterparts of the continuous, regularized
characteristic functions $\chi_i^\varepsilon[\sitesVector,\gVector]$, masses $m_i^\varepsilon[\sitesVector,\gVector]$,
and barycenters $b_i^\varepsilon[\sitesVector,\gVector]$:
\begin{align}
\chi_{i,\alpha}^{\varepsilon,h}[\sitesVector,\gVector]&\coloneqq\chi_{i}^{\varepsilon}[\sitesVector,\gVector](y^\alpha_h)=
\frac{\exp\left(\frac{\g_i-\vert y^\alpha_h-x_i\vert^2}{\varepsilon}\right)}
{\sum_{j=1,\ldots,n}\exp\left(\frac{\g_j-\vert y^\alpha_h-x_j\vert^2}{\varepsilon}\right)},\\
m_i^{\varepsilon,h}[\sitesVector,\gVector] &\coloneqq
\sum_{\alpha\in \mathcal{I}^d} \chi_{i,\alpha}^{\varepsilon,h}[\sitesVector,\gVector]\nu^\alpha, \quad 
b_i^{\varepsilon,h}[\sitesVector,\gVector]\coloneqq
\sum_{\alpha\in \mathcal{I}^d} \frac{y^\alpha_h\; \chi_{i,\alpha}^{\varepsilon,h}[\sitesVector,\gVector]} 
{m_i^{\varepsilon,h}[\sitesVector,\gVector]} \nu^\alpha.
\end{align}
Based on this discretization, we obtain a discrete functional 
\begin{align}\label{funcepsh}
\mathcal{F}_n^{\varepsilon,h}[\sitesVector,\gVector]\coloneqq
\sum_{i=1,\ldots, n} 
m_i^{\varepsilon,h}[\sitesVector,\gVector]\Phi(b_i^{\varepsilon,h}[\sitesVector,\gVector]),
\end{align}
and the usual penalty-enhanced discrete functional 
\begin{align}\label{def:func_discrete}
\mathcal{F}_n^{\eta,\varepsilon,h}[\sitesVector,\gVector]\coloneqq
\sum_{i=1,\ldots, n} 
m_i^{\varepsilon,h}[\sitesVector,\gVector]\Phi(b_i^{\varepsilon,h}[\sitesVector,\gVector])-\eta\mathcal{R}_n^{\varepsilon,h}[\sitesVector,\gVector],
\end{align}
where the discrete penalty term is defined by
\begin{align}\label{def:penalty_discrete}
\mathcal{R}_n^{\varepsilon,h}[\sitesVector,\gVector]\coloneqq\sum_{i=1,\ldots, n}\sum_{\alpha\in \mathcal{I}^d}\vert y^\alpha_h-x_i\vert^2 \chi_{i,\alpha}^{\varepsilon,h}[\sitesVector,\gVector] \nu^\alpha+\sum_{1\leq i,j \leq n \atop {i\neq j}}\frac{m_i^{\varepsilon,h}[\sitesVector,\gVector]\; m_j^{\varepsilon,h}[\sitesVector,\gVector]}{\vert x_i-x_j\vert^2}.
\end{align} 
Now, in analogy to Proposition \ref{prop:gamma_eps} we demonstrate the consistency of entropy regularized and discrete characteristic functions 
and the continuous counterpart. To this end, let us define the $\nu$-almost surely piecewise continuous function $\chi_i^{\varepsilon,h}[\sitesVector,\gVector]\in L^1(\nu)$ as follows: $$\chi_i^{\varepsilon,h}[\sitesVector,\gVector](y)=\chi_{i,\alpha}^{\varepsilon,h}\ \ \text{for } \nu-\text{almost all } y \text{ in a cell interior}  
\stackrel{\circ}{\domain^\alpha_h}.$$
\begin{prpstn}\label{prop:gamma_eps_h}
For $(\varepsilon_N)_N\subset\R^+$, $(\sitesVector^N,\gVector^N)_N\subset\sitesSet\times\R^n$ with $\varepsilon_N\rightarrow 0$, $(\sitesVector^N,\gVector^N)\rightarrow(\sitesVector,\gVector)\in\sitesSet\times\R^n$ for $N\rightarrow\infty$ and for grid sizes $h_N=2^{-N}$ we obtain that 
$
\lim_{N\rightarrow\infty}\chi_i^{\varepsilon_N\mkern-3mu,\mkern1mu h_N}[\sitesVector^N,\gVector^N]=\chi_{\Laguerre_i[\sitesVector,\gVector]},
$
in $L^1(\nu).$
\end{prpstn} 
\begin{proof}
By the triangle inequality, we have that $$\left\Vert\chi_i^{\varepsilon_N\mkern-3mu,\mkern1mu  h_N}[\sitesVector^N,\gVector^N]-\chi_{\Laguerre_i[\sitesVector,\gVector]}\right\Vert_{L^1(\nu)}\leq\left\Vert\chi_i^{\varepsilon_N\mkern-3mu,\mkern1mu  h_N}[\sitesVector^N,\gVector^N]-\chi_i^{\varepsilon_N}[\sitesVector^N,\gVector^N]\right\Vert_{L^1(\nu)}+\left\Vert\chi_i^{\varepsilon_N}[\sitesVector^N,\gVector^N]-\chi_{\Laguerre_i[\sitesVector,\gVector]}\right\Vert_{L^1(\nu)}.$$
In Proposition \ref{prop:gamma_eps}, we already showed that the second term converges to zero.
To verify that the first term vanishes as well in the limit for 
$N\rightarrow\infty$, we estimate the norm of the gradient of $\chi_i^{\varepsilon_N}[\sitesVector^N,\gVector^N]$ given by
\begin{align*}
\nabla_y\chi_i^{\varepsilon_N}[\sitesVector^N,\gVector^N](y)=\chi_i^{\varepsilon_N}[\sitesVector^N,\gVector^N](y)\left(-2\frac{y-x_i^N}{\varepsilon_N}+\frac{\sum_{j=1}^n2(y-x_j^N)\exp\left(\frac{\g_j^N-\vert y-x_j^N\vert^2}{\varepsilon_N}\right)}{\varepsilon_N\sum_{j=1}^n\exp\left(\frac{\g_j^N-\vert y-x_j^N\vert^2}{\varepsilon_N}\right)}\right)\,.
\end{align*}
Let $y\in D$ be in the interior of $\Laguerre_i(\sitesVector,\gVector).$ Then, as in the proof of Proposition \ref{prop:gamma_eps}, there is a $\delta>0$, such that $$\vert y-x_i^N\vert^2-\g_i^N\leq\vert y-x_j^N\vert^2-\g_j^N-\delta$$ for for $j\neq i$ and for sufficiently large $N$. 
Then, one obtains
\begin{align*}
\vert\nabla_y\chi_i^{\varepsilon_N}[\sitesVector^N,\gVector^N](y)\vert&=\frac{2\chi_i^{\varepsilon_N}[\sitesVector^N,\gVector^N](y)}{\varepsilon_N}\left\vert (y-x_i^N)(\chi_i^{\varepsilon_N}[\sitesVector^N,\gVector^N](y)-1)+\frac{\sum_{j\neq i}(y-x_j^N)\exp\left(\frac{\g_j^N-\vert y-x_j^N\vert^2}{\varepsilon_N}\right)}{\sum_{j=1}^n\exp\left(\frac{\g_j^N-\vert y-x_j^N\vert^2}{\varepsilon_N}\right)}\right\vert\\
&\leq \frac{2}{\varepsilon_N}\left(\vert y-x_i^N\vert(1-\chi_i^{\varepsilon_N}[\sitesVector^N,\gVector^N](y))+\sum_{j\neq i}\vert y-x_j^N\vert\exp\left(\frac{-\delta}{\varepsilon_N}\right)\right)\\
&\leq\frac{2}{\varepsilon_N}\left(\vert y-x_i^N\vert\left(\frac{(n-1)\exp(-\delta/\varepsilon_N)}{1+(n-1)\exp(-\delta/\varepsilon_N)}\right)+\sum_{j\neq i}\vert y-x_j^N\vert\exp\left(\frac{-\delta}{\varepsilon_N}\right)\right)\\
&\leq\frac{2}{\varepsilon_N}\exp\left(-\delta/\varepsilon_N\right)\left(\vert y-x_i^N\vert(n-1)+\sum_{j\neq i}\vert y-x_j^N\vert\right)\longrightarrow 0,
\end{align*}
for $N\to\infty$.
Next, let  $y\notin\overline{\Laguerre_i(\sitesVector,\gVector)}.$ Then, it holds 
\begin{align*}
\vert\nabla_y\chi_i^{\varepsilon_N}[\sitesVector^N,\gVector^N](y)\vert
&=\frac{2\chi_i^{\varepsilon_N}[\sitesVector^N,\gVector^N](y)}{\varepsilon_N}\left\vert -(y-x_i^N)+\frac{\sum_{j\neq i}(y-x_j^N)\exp\left(\frac{\g_j^N-\vert y-x_j^N\vert^2}{\varepsilon_N}\right)}{\sum_{j=1}^n\exp\left(\frac{\g_j^N-\vert y-x_j^N\vert^2}{\varepsilon_N}\right)}\right\vert\\
&=\frac{2\chi_i^{\varepsilon_N}[\sitesVector^N,\gVector^N](y)}{\varepsilon_N}\left\vert -(y-x_i^N)+\tilde x\right\vert\leq\frac{2}{\varepsilon_N}\exp(-\delta/\varepsilon_N)\left\vert x_i^N-y+\tilde x\right\vert\rightarrow 0,
\end{align*}
where 
$\tilde x = \sum_{j\neq i}(y-x_j^N)\chi_i^{\varepsilon_N}[\sitesVector^N,\gVector^N](y)$ 
is an element in the convex hull of $(y-x_j^N)_{j=1,\ldots,n}.$ and thus due to the convergence of $(\sitesVector^N)_N$ uniformly bounded. This implies 
that $\vert x_i^N-y+\tilde x\vert$ is uniformly bounded.  

For every $y\in\domain$ in the interior of $\Laguerre_i(\sitesVector,\gVector)$ and sufficiently large $N$ we observe that $y$ is in the cell 
$\domain_{h_N}^{\alpha^N}$ for some multi-indices $\alpha^N$ 
and $\domain_{h_{N_0}}^{\alpha^{N_0}}$ is completely contained in the interior of $\Laguerre_i(\sitesVector,\gVector)$.
There is a $\delta>0$ such that $\vert y-x_i^N\vert^2-\g_i^N\leq\vert y-x_j^N\vert^2-g_j^N-\delta$ holds for $j\neq i$ and all $y\in\domain_{h_N}^{\alpha^N}$.
Hence, the restriction of $\chi_i^{\varepsilon_M}[\sitesVector^M,\gVector^M]$ onto this grid cell for $M\geq N$ 
is a family of uniformly Lipschitz functions, and in particular equicontinuous. It follows that 
\begin{align*}
\left\vert\chi_i^{\varepsilon_N\mkern-3mu,\mkern1mu  h_N}[\sitesVector^N,\gVector^N](y)-\chi_i^{\varepsilon_N}[\sitesVector^N,\gVector^N](y)\right\vert&=\left\vert\chi_i^{\varepsilon_N}[\sitesVector^N,\gVector^N](y_{h_N}^{\alpha^N})-\chi_i^{\varepsilon_N}[\sitesVector^N,\gVector^N](y)\right\vert\\
&\leq\left(\sup_{N\geq N_0}\Lip(\chi_i^{\varepsilon_N}[\sitesVector^N,\gVector^N])\right)\vert y-y_{h_N}^{\alpha^N}\vert\rightarrow 0
\end{align*}

The analogous statement holds for $y\notin\overline{\Laguerre_i(\sitesVector,\gVector)}$ as well.
Thus,  $\chi_i^{\varepsilon_N\mkern-3mu,\mkern1mu  h_N}[\sitesVector^N,\gVector^N]$ converges to $\chi_i^{\varepsilon_N}[\sitesVector^N,\gVector^N]$ $\nu$-almost everywhere and, by the dominated convergence theorem, one obtains the desired convergence in $L^1(\nu)$. 
\end{proof}
We obtain the entropy-regularized, fully discrete analogue of Theorem \ref{thm:maximizer}.
\begin{thrm} 
Let us assume that $\varepsilon, h>0$, and $\Phi\in C(\domain)$ with maximum $\bar \Phi$ on $D$.
Then, for given number of sites $l\in\N$ there exists an $n\leq l$, 
such that a maximizer $(\sitesVector^*,\gVector^*)$ of $\mathcal{F}_n^{\eta,\varepsilon,h}$ exists  
with  $m_i^{\varepsilon,h}[\sitesVector^*,\gVector^*]>0$ for $i=1,\ldots,n$
and  $\mathcal{F}_n^{\eta,\varepsilon,h}(\sitesVector^*,\gVector^*)\geq\sup_{(\sitesVector,\gVector)}\mathcal{F}_l^{\eta,\varepsilon,h}(\sitesVector,\gVector)$.
\end{thrm}
\begin{proof}
The proof is completely analogous to the proof of Theorem \ref{thm:maximizer}.
\end{proof}

Next, we investigate the 
convergence of maximizers of the fully discrete, entropy regularized functional $\mathcal{F}_n^{\eta,\varepsilon,h}$ 
given in  \eqref{def:func_discrete} to a
maximizer of the original functional $\mathcal{F}_n^\eta$ defined in \eqref{eq:objective3} for both $\varepsilon\to 0$ and $h\to 0$.
\begin{thrm}
Let $\Phi\in C(\domain),$ $l\in\N$ and $\eta>0$. For each $N\in\N$, consider maximizers $(\sitesVector^N,\gVector^N)\in\mathcal{X}_{\beta(N)}\times\R^{\beta(N)}$, where $\beta:\N\rightarrow\{1,\ldots,n\}$, so that $\Func_{\beta(N)}^{\eta,\varepsilon_N\!,h_N}(\sitesVector^N,\gVector^N)\geq\max_{k=1,\ldots,n}\{\sup_{(\sitesVector,\gVector)\in\mathcal{X}_k\times\R^k}\Func_k^{\eta,\varepsilon_N\!,h_N}(\sitesVector,\gVector)\}$. Then, there is $n\in\{1,\ldots,l\}$ such that up to the selection of a subsequence $(\sitesVector^{N},\gVector^{N})_{N\in\N}\subset\sitesSet\times\R^n$  converges to a limit $(\sitesVector^*,\gVector^*)\in\sitesSet\times\R^{n}$. Furthermore, it holds $$\Func_{n}^{\eta}(\sitesVector^*,\gVector^*)\geq\max_{k=1,\ldots,l}\left\{\sup_{(\sitesVector,\gVector)\in\mathcal{X}_k\times\R^k}\Func_k^{\eta}(\sitesVector,\gVector)\right\}.$$
\end{thrm}
\begin{proof}
The proof of the first statement is along the same lines as the proof of Theorem \ref{thm:maximizer}. For the last statement, notice that 
\begin{align*}
b_i^{\varepsilon_N\!,h_N}[\sitesVector^N,\gVector^N]&=
\sum_{\alpha\in \mathcal{I}^d} \frac{y^\alpha_h\; \chi_{i,\alpha}^{\varepsilon_N\!,h_N}[\sitesVector^N,\gVector^N]} 
{m_i^{\varepsilon_N\!,h_N}[\sitesVector^N,\gVector^N]} \nu^\alpha=\int_\domain\frac{g^N(y)\; \chi_{i}^{\varepsilon_N\!,h_N}[\sitesVector^N,\gVector^N]} 
{m_i^{\varepsilon_N\!,h_N}[\sitesVector^N,\gVector^N](y)}\d\nu(y)\\
\mathcal{R}_n^{\varepsilon_N\!,h_N}[\sitesVector^N,\gVector^N]&=\sum_{i=1,\ldots, n}\sum_{\alpha\in \mathcal{I}^d}\vert y^\alpha_h-x_i^N\vert^2 \chi_{i,\alpha}^{\varepsilon_N\!,h_N}[\sitesVector^N,\gVector^N]\nu^\alpha+\sum_{1\leq i,j \leq n \atop {i\neq j}}\frac{m_i^{\varepsilon_N\!,h_N}[\sitesVector^N,\gVector^N]\; m_j^{\varepsilon_N\!,h_N}[\sitesVector^N,\gVector^N]}{\vert x_i^N-x_j^N\vert^2}\\
&=\sum_{i=1,\ldots, n}\int_\domain f^N(y) \chi_{i}^{\varepsilon_N\!,h_N}[\sitesVector^N,\gVector^N](y)\d\nu(y)+\sum_{1\leq i,j \leq n \atop {i\neq j}}\frac{m_i^{\varepsilon_N\!,h_N}[\sitesVector^N,\gVector^N]\; m_j^{\varepsilon_N\!,h_N}[\sitesVector^N,\gVector^N]}{\vert x_i^N-x_j^N\vert^2},
\end{align*}
where $f^N(y)\coloneqq y^\alpha_h$ for $y\in\stackrel{\circ}{\domain^\alpha_h}$, and $g^N(y)\coloneqq\vert y^\alpha_h-x_i^N\vert^2$ for $y\in\stackrel{\circ}{\domain^\alpha_h}$ are the piecewise constant approximations of the functions $y\mapsto y$ and $y\mapsto\vert y-x_i^N\vert^2$ consistent with our grid cells. The rest of the proof is then completely analogous to the proof of the respective statement in Theorem \ref{thm:max_conv_eps}.
\end{proof}

To numerically implement the maximization of \ref{infosec_pob_reg} via a gradient ascent approach we 
need to compute the gradient of $\mathcal{F}_n^\varepsilon$.  We specifically 
obtain the following derivatives for spatially discrete quantities using $z=x_k$ or $z=\g_k$ for $k=1,\ldots, n$
\begin{align*}
\partial_{z}\mathcal{F}_n^{\varepsilon,h}[\sitesVector,\gVector] &= 
\sum_{j=1,\ldots, n} 
\partial_z m_j^{\varepsilon,h}[\sitesVector,\gVector] \Phi(b_j^{\varepsilon,h}[\sitesVector,\gVector])
+  m_j^{\varepsilon,h}[\sitesVector,\gVector] \partial_z \Phi(b_j^{\varepsilon,h}[\sitesVector,\gVector]),\\
\partial_{x_k}\mathcal{R}_n^{\varepsilon,h}[\sitesVector,\gVector] &=
2\sum_{\alpha\in \mathcal{I}^d}(x_k-y^\alpha_h) \chi_{k,\alpha}^{\varepsilon,h}[\sitesVector,\gVector]\nu^\alpha+\sum_{i=1,\ldots, n}\sum_{\alpha\in \mathcal{I}^d}\vert y^\alpha_h-x_i\vert^2\partial_{x_k}\chi_{i,\alpha}^{\varepsilon,h}[\sitesVector,\gVector]\nu^\alpha\\
&+\sum_{1\leq i,j \leq n \atop {i\neq j}}\left(\frac{\partial_{x_k}m_i^{\varepsilon,h}[\sitesVector,\gVector]m_j^{\varepsilon,h}[\sitesVector,\gVector]}{\vert x_i-x_j\vert^2}+\frac{m_i^{\varepsilon,h}[\sitesVector,\gVector]\partial_{x_k}m_j^{\varepsilon,h}[\sitesVector,\gVector]}{\vert x_i-x_j\vert^2}\right.\\
&-\left.2\frac{m_i^{\varepsilon,h}[\sitesVector,\gVector]m_j^{\varepsilon,h}[\sitesVector,\gVector]}{\vert x_i-x_j\vert^4}((x_i-x_j)\delta_{ik}+(x_j-x_i)\delta_{jk})\right) ,\\
\partial_{\g_k}\mathcal{R}_n^{\varepsilon,h}[\sitesVector,\gVector] &=
\sum_{i=1,\ldots, n}\sum_{\alpha\in \mathcal{I}^d}\vert y^\alpha_h-x_i\vert^2\partial_{\g_k}\chi_{i,\alpha}^{\varepsilon,h}[\sitesVector,\gVector] \nu^\alpha\\
&+\sum_{1\leq i,j \leq n \atop {i\neq j}}\left(\frac{\partial_{\g_k}m_i^{\varepsilon,h}[\sitesVector,\gVector]m_j^{\varepsilon,h}[\sitesVector,\gVector]}{\vert x_i-x_j\vert^2}+\frac{m_i^{\varepsilon,h}[\sitesVector,\gVector]\partial_{\g_k}m_j^{\varepsilon,h}[\sitesVector,\gVector]}{\vert x_i-x_j\vert^2}\right) ,\\
\partial_z \Phi(b_j^{\varepsilon,h}[\sitesVector,\gVector]) &= \nabla \Phi(b_j^{\varepsilon,h}[\sitesVector,\gVector]) 
\cdot \partial_z b_j^{\varepsilon,h}[\sitesVector,\gVector], \\
\partial_z m_j^{\varepsilon,h}[\sitesVector,\gVector] &= 
\sum_{\alpha\in \mathcal{I}^d} \partial_z  \chi_{j,\alpha}^{\varepsilon,h}[\sitesVector,\gVector] \nu^\alpha, \\
\partial_z b_j^{\varepsilon,h}[\sitesVector,\gVector] &=
\frac{1}{m_j^{\varepsilon,h}[\sitesVector,\gVector]} 
\sum_{\alpha\in \mathcal{I}^d} y^\alpha_h \partial_z  \chi_{j,\alpha}^{\varepsilon,h}[\sitesVector,\gVector] \nu^\alpha  
- \frac{b_j^{\varepsilon,h}[\sitesVector,\gVector]  \partial_z m_j^{\varepsilon,h}[\sitesVector,\gVector]}{m_j^{\varepsilon,h}[\sitesVector,\gVector]^2},\\
\partial_{x_k} \chi_{j,\alpha}^{\varepsilon,h}[\sitesVector,\gVector] &=
\tfrac{2 x_j}{\varepsilon} \left(\chi_{j,\alpha}^{\varepsilon,h}[\sitesVector,\gVector] -\delta_{kj}\right)  \chi_{j,\alpha}^{\varepsilon,h}[\sitesVector,\gVector], \\
\partial_{\g_k} \chi_{j,\alpha}^{\varepsilon,h}[\sitesVector,\gVector] &=
-\tfrac{1}{\varepsilon} \left(\chi_{j,\alpha}^{\varepsilon,h}[\sitesVector,\gVector] -\delta_{kj}\right)  \chi_{j,\alpha}^{\varepsilon,h}[\sitesVector,\gVector].
\end{align*}

\section{Numerical experiments}\label{sec:numerics}

Now, we will apply the above derived method to compute the optimal partition of the unit  cube $[0,1]^2$ for given $n$, a function $\Phi$ and a probability density $\nu$. In the presented numerical results, we assume $\nu$ to be the Lebesgue measure on the unit cube. Moreover, for every numerical experiment in this and the next section, we made use of the sparse multi-scale algorithm by Schmitzer \cite{Sc19} and the implementation in \cite{ChFe21} to efficiently compute the Sinkhorn iterations. To optimize the parameters of the power diagrams, we used the \emph{Adam} optimizer (cf. \cite{Ki14}) as a stochastic gradient ascent method. Unless otherwise explicitly stated, we use $\eta=0$, $N=256$ (i.e. grid size $h=1/256$), and initialize our algorithm with $n=12$ sites/weights for each of the experiments in this and the next section. On each plot in this section, the barycenters (triangles) and sites (circles) are plotted, for visualisation purposes the latter ones only if inside the unit cube. On the left of each plot, the respective values of the component of the $\g$ vector is also plotted.

In Figure~\ref{fig:blur_param} we show the dependence of the optimal numerical solution on the entropy parameter $\varepsilon$. The interfaces between the Laguerre cells become fuzzier with increasing $\varepsilon$, whereas the structure of the optimal solutions does not change much.
\begin{figure}[ht]
\resizebox{\linewidth}{!}{
\tikzstyle{frame} = [line width=1.8pt, draw=red,inner sep=0.01em]
\begin{tikzpicture}
\begin{scope}[scale=1.0]

\begin{scope}[shift={(0,-11.02)}]
\node[anchor=south west] at (0.,0.) 
{\includegraphics[width=0.2\linewidth]{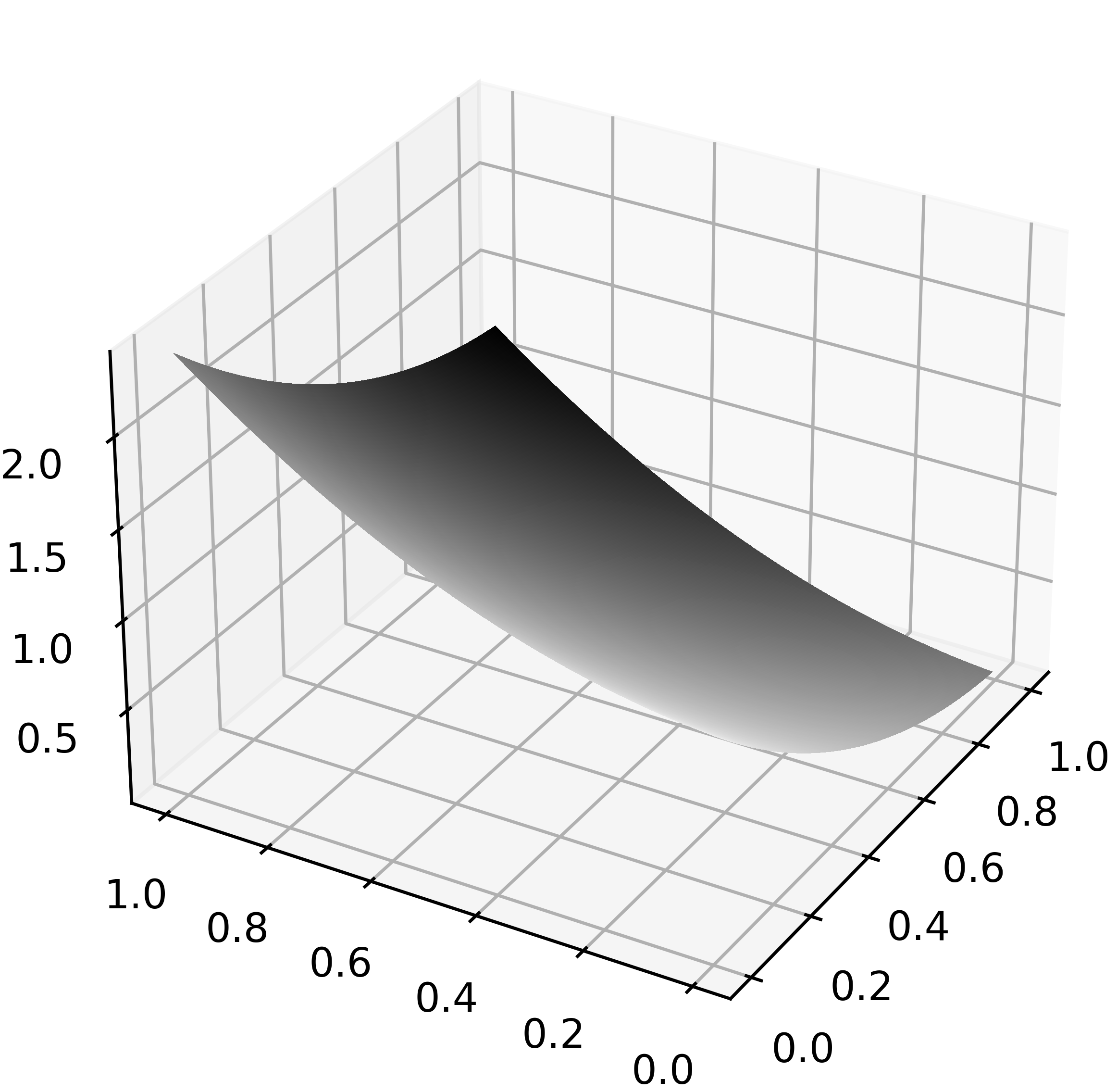}};
\node[anchor=south west] at (3.8,0.) 
{\includegraphics[width=0.2\linewidth]{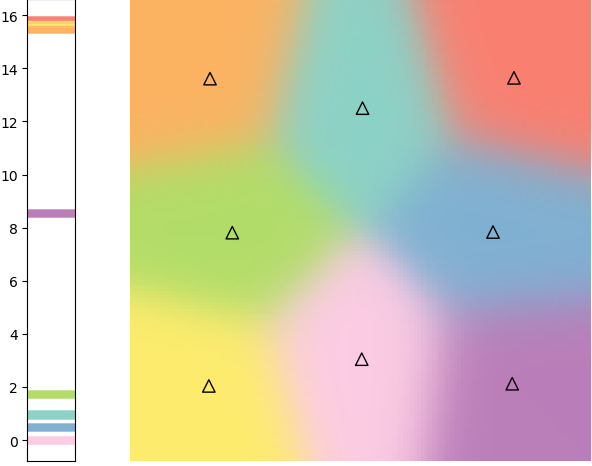}};
\node[anchor=south west] at (7.6,0.)
{\includegraphics[width=0.2\linewidth]{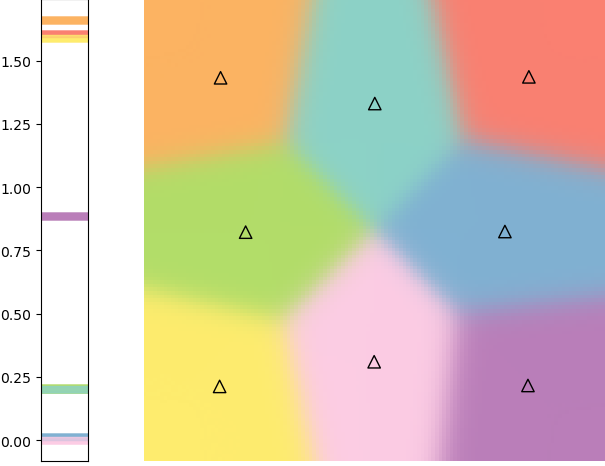}};
\node[anchor=south west] at (11.4, 0.)
{\includegraphics[width=0.2\linewidth]{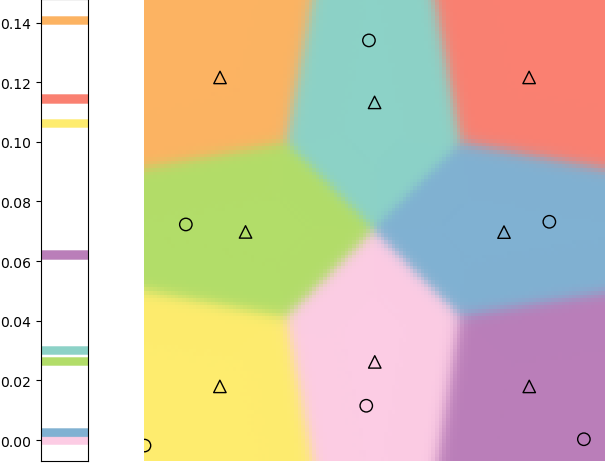}};
\node[anchor=south west] at (15.2,0.)
{\includegraphics[width=0.2\linewidth]{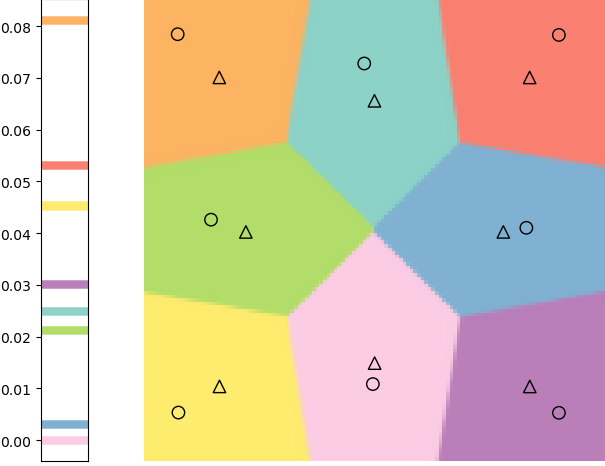}};
\end{scope}
\end{scope}
\end{tikzpicture}
}
\caption{Convergence of the optimal power diagram for different entropy parameters $\varepsilon=25,5,1,0.2.$ (four right-most plots), where the blur parameter values are given in units of $1/N$, for $N=128$. The function $\Phi$ is plotted on the left-most panel.}
\label{fig:blur_param}
\end{figure} 

Next, we plot in Figure~\ref{fig:counterexample} the optimal configurations for different values of $\eta$ 
for a function $\Phi$ (left), which has global/local maxima of equal value at the points $(\tfrac12, \tfrac14), (\tfrac34, \tfrac34)$, and $(\tfrac14, \tfrac34)$.  We know that the solution for $\eta=0$ is not contained in the space of power diagrams, but
we observe for $\eta \to 0$ the convergence of the solutions.
Indeed, for large values of $\eta$ a single cell has positive mass and no other cells contribute to the cost.  
For vanishing $\eta$, two additional cells appear as optimal configurations and the three cells meet in a triple-point converging to $(\tfrac12, \tfrac12)$ with one angle  converging to $\pi$.
 
 \begin{figure}[ht]
\resizebox{\linewidth}{!}{
\tikzstyle{frame} = [line width=1.8pt, draw=red,inner sep=0.01em]
\begin{tikzpicture}
\begin{scope}[scale=1.0]

\begin{scope}[shift={(0,-11.02)}]
\node[anchor=south west] at (0.,0.) 
{\includegraphics[width=0.2\linewidth]{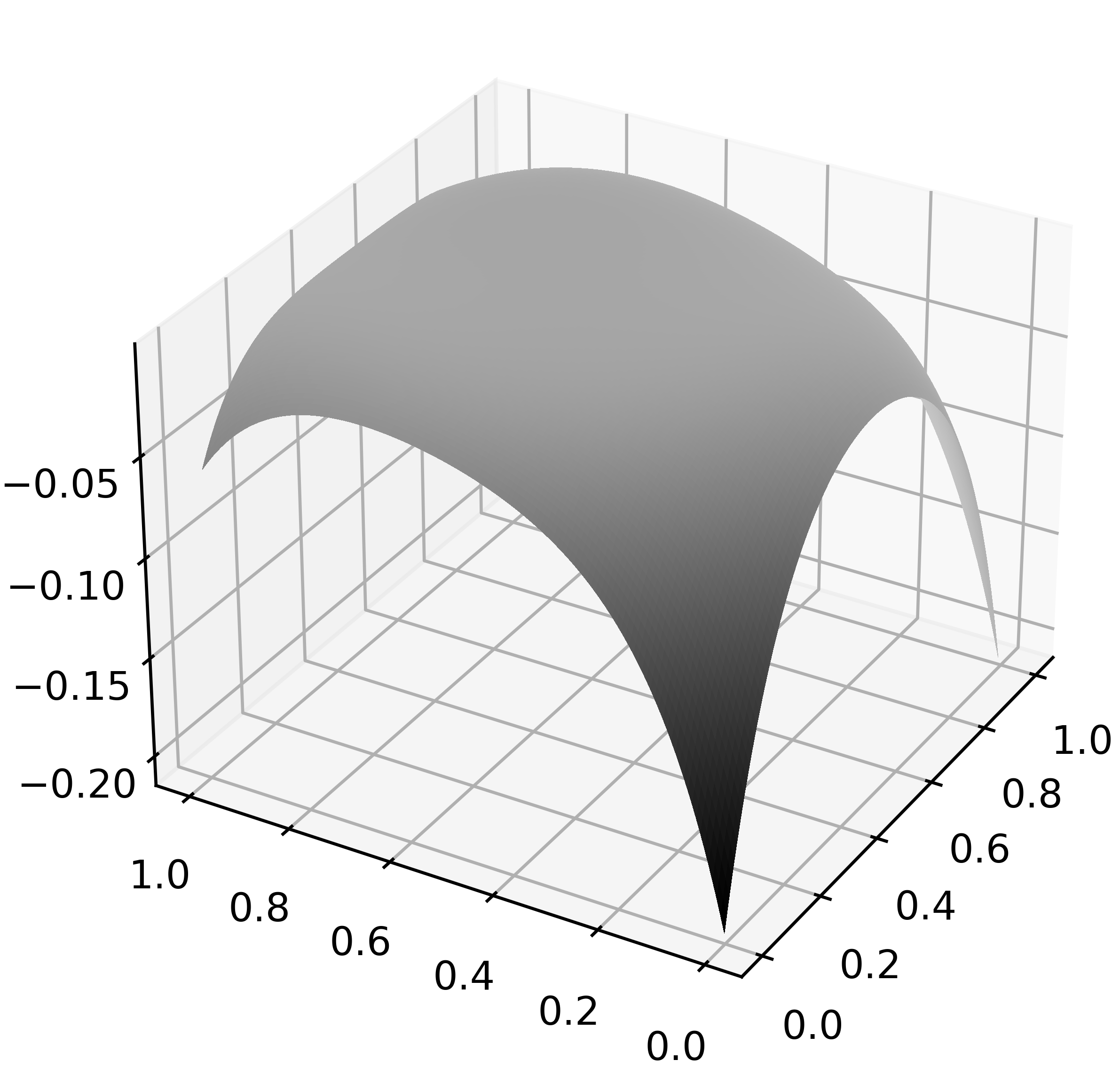}};
\begin{scope}[shift={(3.8,0.)}]
\node[anchor=south west] at (0.,0.) 
{\includegraphics[width=0.2\linewidth]{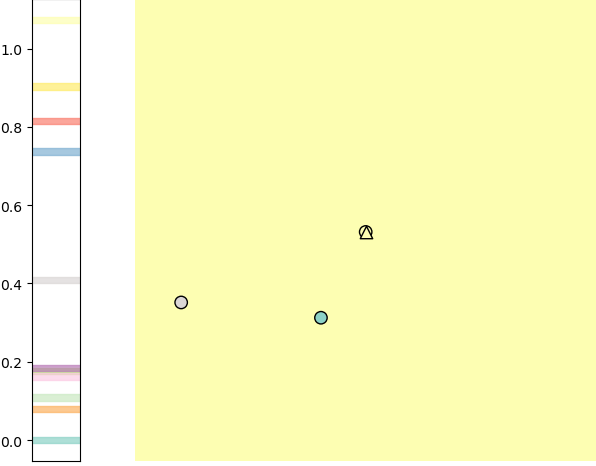}};
\node[anchor=south west] at (3.8,0.) 
{\includegraphics[width=0.2\linewidth]{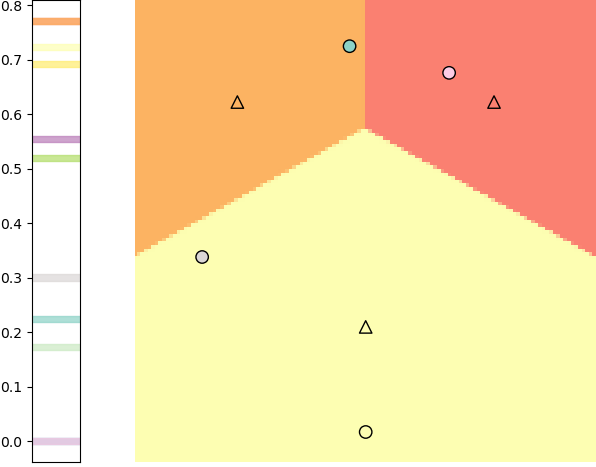}};
\node[anchor=south west] at (7.6,0.)
{\includegraphics[width=0.2\linewidth]{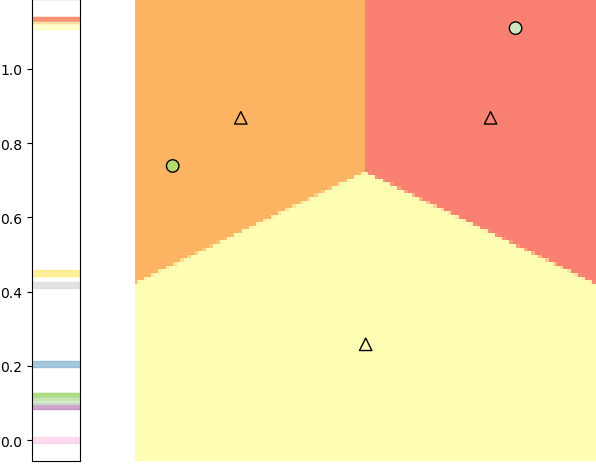}};
\node[anchor=south west] at (11.4, 0.)
{\includegraphics[width=0.2\linewidth]{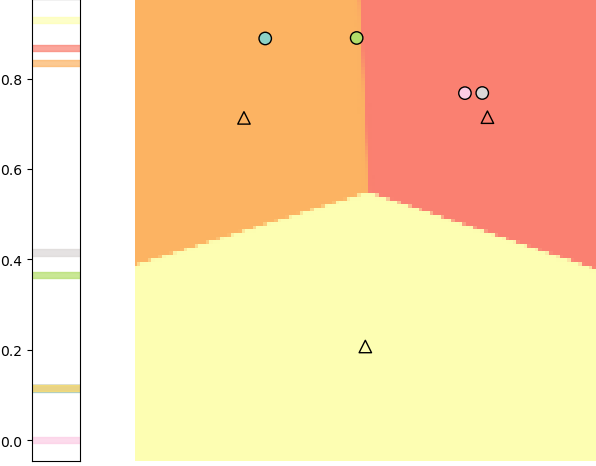}};
\node[anchor=south west] at (15.2,0.)
{\includegraphics[width=0.2\linewidth]{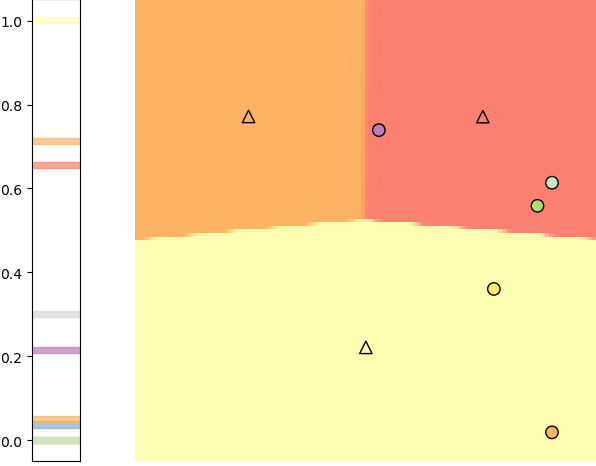}};
\end{scope}
\end{scope}
\end{scope}
\end{tikzpicture}
}
\caption{Convergence of the optimal power diagram for a function $\Phi$ (column (i)) with global maxima at $(0.25, 0.75), (0.75,0.75)$ and $(0.5,0.25)$ for regularization parameter $\eta=1\mathrm{e}{-1}, 1\mathrm{e}{-2}, 1\mathrm{e}{-3}, 1\mathrm{e}{-4}, 1\mathrm{e}{-5}$ (columns (ii)-(vi)).}
\label{fig:counterexample}
\end{figure} 

Finally, we check how our algorithm deals with fusing/pushing cells away, when the optimal solution requires a smaller number of sites/weights than what the algorithm was initialized with. To this end, in Figure~\ref{fig:concave_example} 
we consider a concave function $\Phi$, that has the trivial partition as solution, and plot some iterations of our algorithm that show the proper recovery of this solution, i.e. cells disappear by either pushing the respective sites away (orange, green cells), or by increasing the difference between the $\g$-values of the purple cell and those of the respective cells (beige, blue).

\begin{figure}[ht]
\resizebox{\linewidth}{!}{
\tikzstyle{frame} = [line width=1.8pt, draw=red,inner sep=0.01em]
\begin{tikzpicture}
\begin{scope}[scale=1.0]

\begin{scope}[shift={(0,-11.02)}]
\node[anchor=south west] at (0.,0.) 
{\includegraphics[width=0.2\linewidth]{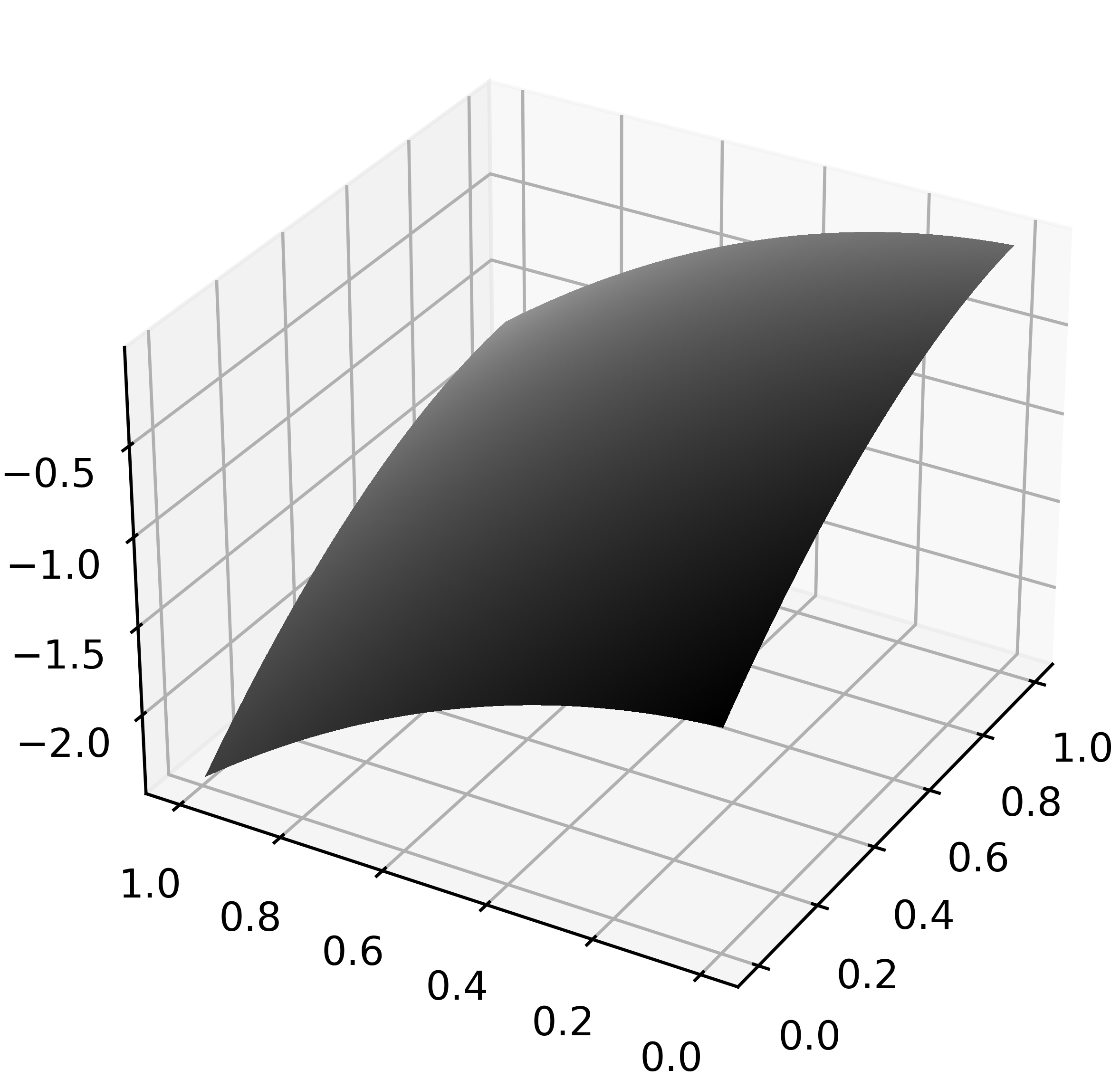}};
\node[anchor=south west] at (3.8,0.) 
{\includegraphics[width=0.2\linewidth]{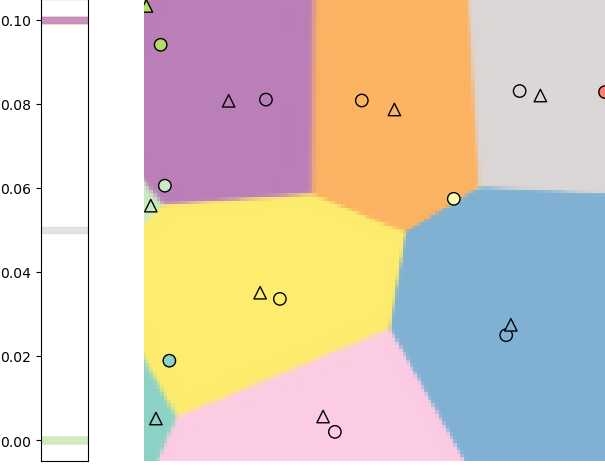}};
\node[anchor=south west] at (7.6,0.)
{\includegraphics[width=0.2\linewidth]{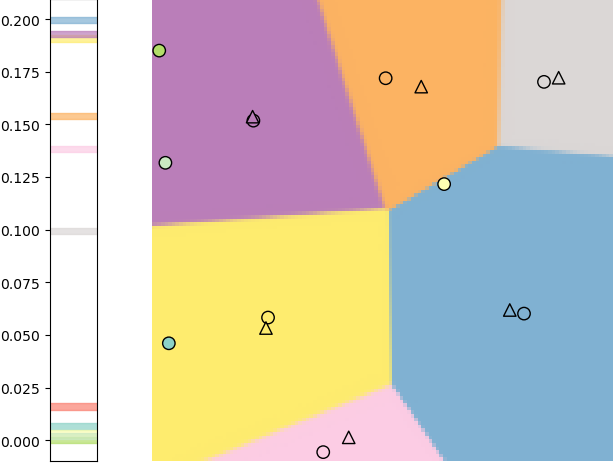}};
\node[anchor=south west] at (11.4, 0.)
{\includegraphics[width=0.2\linewidth]{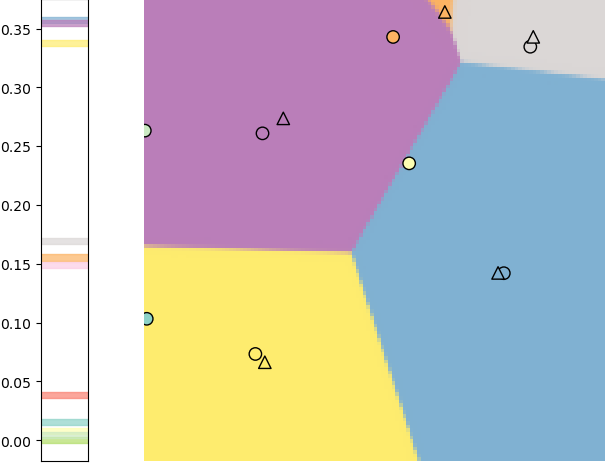}};
\node[anchor=south west] at (15.2,0.)
{\includegraphics[width=0.2\linewidth]{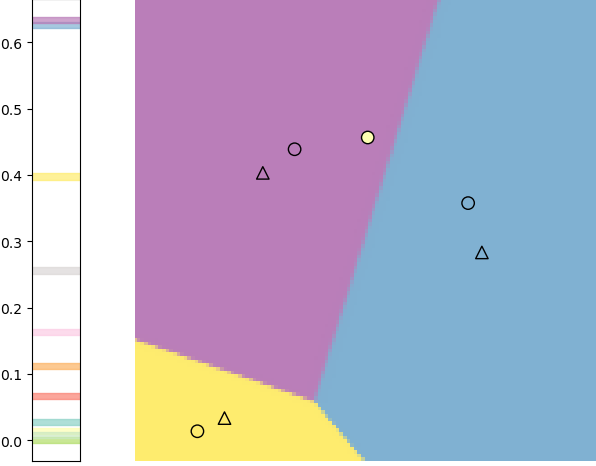}};
\node[anchor=south west] at (19.0,0.)
{\includegraphics[width=0.2\linewidth]{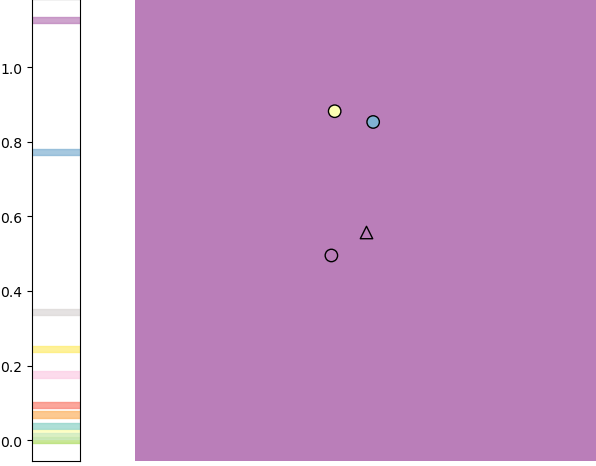}};
\end{scope}
\end{scope}
\end{tikzpicture}
}
\caption{Convergence of the optimal power diagram for a concave function $\Phi$ (left). One sees how the algorithm automatically pushes cells outside of the relevant unit square to enact the trivial solution. Plotted are the computed power diagrams after iterations $it=1, 2, 4, 8, 16$ (five right-most plots).}
\label{fig:concave_example}
\end{figure} 

\section{Application to information design}\label{sec:application}
Finally, we consider the application of our algorithm to a particular problem in information design: the monopolist's problem. In this problem, a seller (sender) can disclose information about the qualities of the products she sells.
There are two objects available individually at fixed prices $p_{1},p_{2}$, and the bundle of both objects can be bought at fixed price $p_3=p_1+p_2+\Delta$ for some bundling ``surcharge'' $\Delta > 0$ or ``discount'' $\Delta < 0$.
A consumer (receiver) has valuations $\mathbf{v}=(v_{1},v_{2})$ per unit of quality, distributed according to the Lebesgue measure on the unit square $[0,1]^{2}$. 
The quality of the objects $\mathbf{q}=(q_{1},q_{2})$ is distributed on $[\underline{q},\overline{q}]^{2}$ according to a measure $\mu $ that is absolutely continuous with respect to the Lebesgue measure. Throughout this section, we assume that $\mu$ is the (re-scaled) Lebesgue measure; the general case works analogously. The realized qualities are known to the seller but not to the buyer, and the realized valuations are known to the buyer but not to the seller.

We first consider a buyer with unit demand and then consider a buyer who demands more than one object and has additive valuations.
A buyer has \emph{unit demand} if she values at most one object. If the seller provides a signal about the qualities such that the buyer believes the expected qualities of the objects are $(q_{1},q_{2})$, then a buyer with unit demand buys only good $i=1,2$ if
\[
    q_{i}v_{i}-p_{i}> \max \{0,\ q_{-i}v_{-i}-p_{-i}\}
\]
and buys nothing if 
\[
    \max_{i}\{q_{i}v_{i}-p_{i}\}<0.
\]
Here, we ignore ties, which have probability zero of occurring.
If a buyer with \emph{additive valuations} buys only object $i$, her payoff is $q_i v_i-p_i$ but if she buys both objects her payoff is $q_1 v_1 + q_2 v_2 -p_3$.
Therefore, a buyer with additive valuations buys only object $i$ if
\[
    q_{i}v_{i}-p_{i}> \max \{0,\ q_{-i}v_{-i}-p_{-i}, \ q_1 v_1+q_2v_2 -p_3\},
\]
buys both objects if
\[
    q_1 v_1+q_2v_2 -p_3> \max \{0,\ q_{1}v_{1}-p_{1}, \ q_{2}v_{2}-p_{2}\},
\]
and buys nothing otherwise.

For fixed expected qualities $(q_{1},q_{2}),$ let $C_{i}(q_{1},q_{2})$ be the probability assigned to the set of valuations for which the buyer only buys object $i=1,2,$ let $C_3(q_1,q_2)$ denote the probability assigned to the set of valuations for which the consumer buys both objects, and let $C_{0}(q_{1},q_{2})$ be the probability assigned to the set of valuations where the consumer buys nothing. Recall that these depend on whether we consider a consumer with unit demand or with additive valuations, and are computed from the uniform distribution of valuations: these are the areas of the respective convex polygons defined by the above inequalities.
The seller's revenue is then given by 
\[
    R(q_{1},q_{2})=\sum_{i=1}^{3}p_{i}C_{i}(q_{1},q_{2}).
\]%
 
Recalling the characterization of exposed points, the seller chooses
a Laguerre diagram $\pi =\{D^{1},D^{2},\ldots,D^{n}\}$ of $[\underline{q},%
\overline{q}]^{2}$ with respective barycenters $\mathbf{q}^{j}=(q_{1}^{j}%
\mathbf{,}q_{2}^{j}),\ j=1,2,\ldots,n.$ In other words, for each realization of
qualities the seller reveals to the consumer only to which cell the qualities belong.
The consumer then updates her belief about the expected qualities to the barycenter $\mathbf{q}^{j}$ of the cell $D^{j}$ which contains the true qualities.

The designer's expected revenue is then given by 
\[
\sum_{j=1}^{n}\mu (D^{j})R(q_{1}^{j}\mathbf{,}q_{2}^{j}) 
\]
and the designer chooses a Laguerre diagram to maximize this expected revenue. 

\paragraph{Unit demand.} We first consider a buyer with unit demand. 
Figure~\ref{fig:infodisc_unit_demand} shows that the seller discloses only a coarse signal about the qualities of the products. Each cell corresponds to one expected quality pair of the products which we indicate by a triangle in the figure. 
There are at most 4 cells in the optimal signal:  These cells can be roughly interpreted as corresponding to different quality pairs: (1) the orange cell corresponds to both products being of low expected quality, (2) the purple cell corresponds to product 1 being of low quality and product 2 being of high quality, (3) the red cell corresponds to product 1 being of high quality and product 2 being of low quality, and (4) the green cell corresponds to both products being of high  quality.
In the orange cells, the expected qualities are so low that the buyer never buys either of the products, independent of their valuations. In the red and purple cells, the expected quality of the lower-quality object is so low, that the buyer will never buy the lower-quality object; the buyer will either buy the higher-quality object or buy nothing at all.
As the price of product 2 increases, in the optimal information policy, a signal indicating a high quality of product 2 becomes more informative in that it indicates a higher expected quality. This offsets the higher price and still induces some buyer types to purchase the more expensive product.  If the price of product 2 becomes too high, no types will buy product 2 and the optimal Laguerre diagram has only two cells, as the right-most panel illustrates.

Under the optimal Laguerre diagram, the seller does not provide full information to the buyer even though that would raise efficiency. By revealing only imprecise information, the buyer's information rents are reduced and the seller's revenue raised. To evaluate the  benefit of choosing an optimal information policy induced by a Laguerre diagram, we consider additional, non-optimal information policies, as benchmarks. 
Table~\ref{tab:obj_unit_demand} specifies the revenue 
generated under various information policies. It shows in row
(i) the different values of the price $p_2$, in row 
(ii) the revenue $R_{opt}$ induced by the optimal Laguerre diagram partition computed with our algorithm, 
in row (iii) the revenue $R(1, 1)$  for an 
information policy where no information is given to the buyer, and
in row (iv) the revenue for an additional benchmark information policy based on a partition generated with Lloyd's algorithm. Here, for the same number of cells as in the computed optimal Laguerre diagram, we generate with Lloyd's algorithm a Laguerre diagram partition which imposes the barycenter of each cell to coincide with the respective site, cf.~\cite{Ll82}.  The respective Lloyd diagrams 
used for benchmarking are exemplified in Figure~\ref{fig:RevenuePlot2}. 
Row (v) of Table~\ref{tab:obj_unit_demand} shows the revenue $\mathbb{E}(R)$ in case of the full information policy, 
and in row (vi) the percentage increase in revenue of the optimal power diagram policy compared to the full information policy is displayed. As one can see, optimal information design creates significant value to the seller: in this example it increases profits relative to full information in excess of $10\%$.

Figure~\ref{fig:infodisc_unit_demand_qminus} shows the optimal Laguerre diagrams for prices $p_1=1$ and $p_2=1.25$ and various values of the lower bound on the quality. As the lower bound increases, the optimal signals become less likely to contain significant information about the quality of object 2. The corresponding revenue and the revenue under alternative benchmarks is shown in Table~\ref{tab:obj_unit_demand_qminus}.
\begin{table}[ht]
\begin{center}
\begin{tabular}{ c | c | c | c | c | c }                
$p_2$ & 1 & 1.25 & 1.5 & 1.75 & 2  \\ 
\hline \hline 
$\mathbf{R_{opt}}$ & $\mathbf{0.3153}$ & $\mathbf{0.2648}$ & $\mathbf{0.2175}$ & $\mathbf{0.1839}$ & $\mathbf{0.1716}$\\ \hline                   
$R(1, 1)$ & $0.0000$ & $0.0000$ & $0.0000$ & $0.0000$ & $0.0000$\\ \hline
$R_{Lloyd}$ & $0.3056$ & $0.2543$ & $0.1667$ & $0.1667$ & $0.1667$\\ \hline
$\mathbb{E}(R)$ & $0.2833$ & $0.2407$ & $0.1977$ & $0.1657$ & $0.1534$\\ \hline
$pp (\%)$ & $11.30$ & $10.01$ & $10.01$ & $10.98$ & $11.86$
\end{tabular}
\caption{Monopolist's problem with unit demand, price $p_1=1$, quality boundaries $\underline{q}=0$ and $\overline{q}=2$ for different prices $p_2$ (first row) under selected information policies. \label{tab:obj_unit_demand}}
\end{center} \end{table}
\begin{figure}[h]
\resizebox{\linewidth}{!}{
\tikzstyle{frame} = [line width=1.8pt, draw=red,inner sep=0.01em]
\begin{tikzpicture}
\begin{scope}[scale=1.0]
\begin{scope}[shift={(0,0)}]
\node[anchor=south west] at (0.,0.) 
{\includegraphics[width=0.2\linewidth]{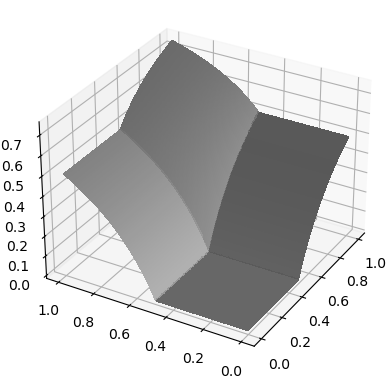}};
\node at (1.0,0.35) {\scriptsize{$q_1$}};
\node at (3.2,0.55) {\scriptsize{$q_2$}};
\node[anchor=south west] at (3.8,0.) 
{\includegraphics[width=0.2\linewidth]{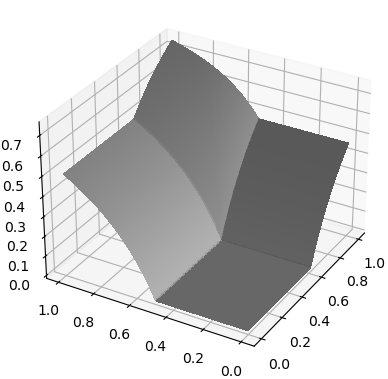}};
\node at (4.8,0.35) {\scriptsize{$q_1$}};
\node at (7.0,0.55) {\scriptsize{$q_2$}};
\node[anchor=south west] at (7.6,0.)
{\includegraphics[width=0.2\linewidth]{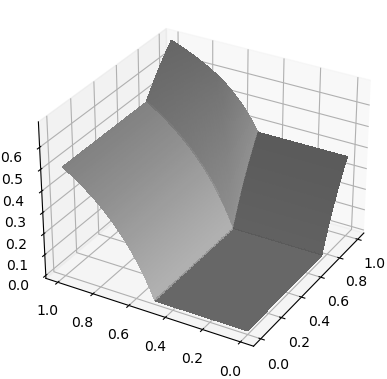}};
\node at (8.6,0.35) {\scriptsize{$q_1$}};
\node at (10.8,0.55) {\scriptsize{$q_2$}};
\node[anchor=south west] at (11.4, 0.)
{\includegraphics[width=0.2\linewidth]{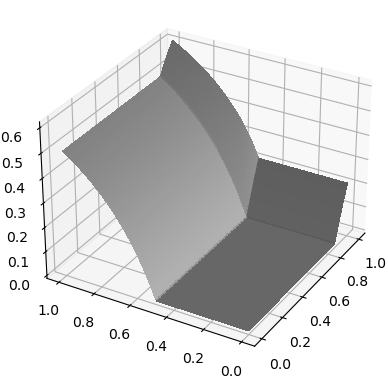}};
\node at (12.4,0.35) {\scriptsize{$q_1$}};
\node at (14.6,0.55) {\scriptsize{$q_2$}};
\node[anchor=south west] at (15.2,0.)
{\includegraphics[width=0.2\linewidth]{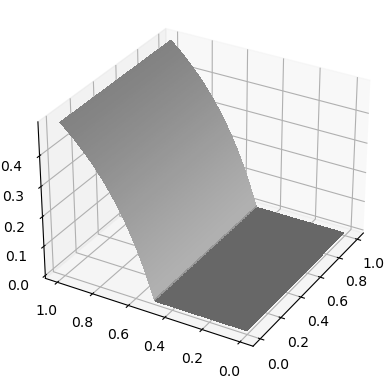}};
\node at (16.2,0.35) {\scriptsize{$q_1$}};
\node at (18.4,0.55) {\scriptsize{$q_2$}};
\end{scope}

\begin{scope}[shift={(0,-3.62)}]
\node[anchor=south west] at (0.,0.) 
{\includegraphics[width=0.2\linewidth]{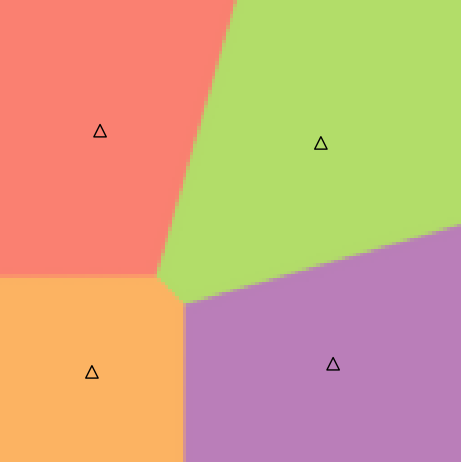}};
\node at (-0.1,1.8) {\small{$q_1$}};
\node at (1.9,-0.1) {\small{$q_2$}};
\node[anchor=south west] at (3.8,0.) 
{\includegraphics[width=0.2\linewidth]{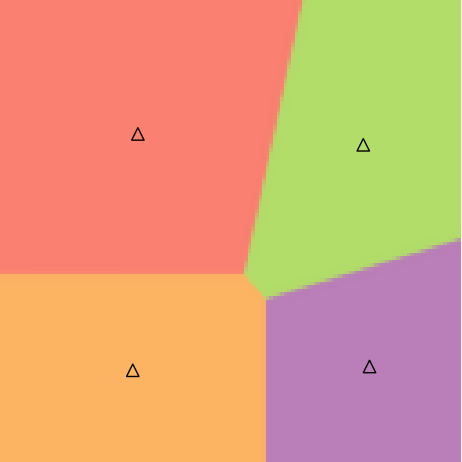}};
\node at (5.7,-0.1) {\small{$q_2$}};
\node[anchor=south west] at (7.6,0.)
{\includegraphics[width=0.2\linewidth]{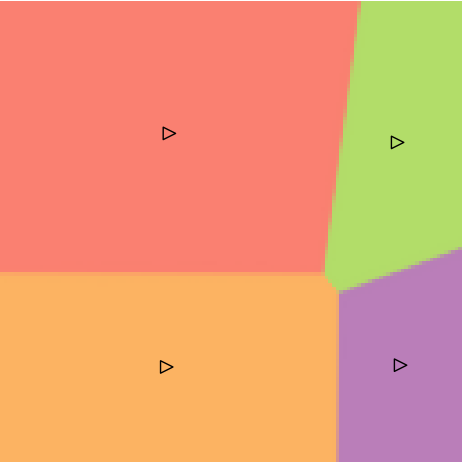}};
\node at (9.5,-0.1) {\small{$q_2$}};
\node[anchor=south west] at (11.4, 0.)
{\includegraphics[width=0.2\linewidth]{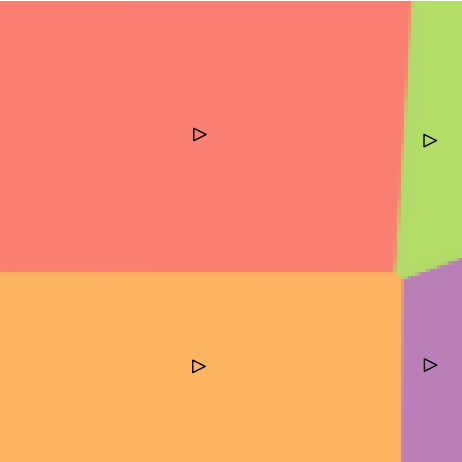}};
\node at (13.3,-0.1) {\small{$q_2$}};
\node[anchor=south west] at (15.2,0.)
{\includegraphics[width=0.2\linewidth]{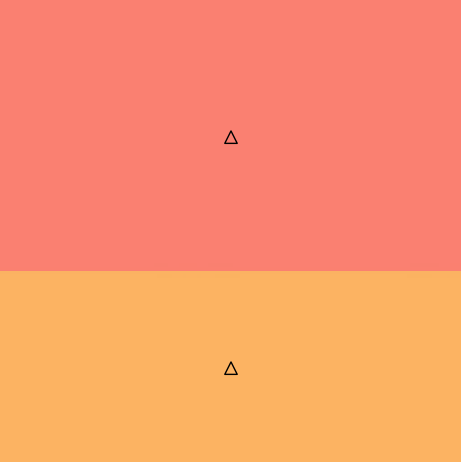}};
\node at (17.1,-0.1) {\small{$q_2$}};
\end{scope}
\end{scope}
\end{tikzpicture}
}
\caption{Optimal configurations for the monopolist's problem with unit demand, with prices $p_1=1$ and $p_2=1, 1.25, 1.5, 1.75, 2$ (second row, from left to right) and quality boundaries $\underline{q}=0$ and $\overline{q}=2$. 
The respective revenue function $R$ for each case is plotted in the first row.}
\label{fig:infodisc_unit_demand}
\end{figure}
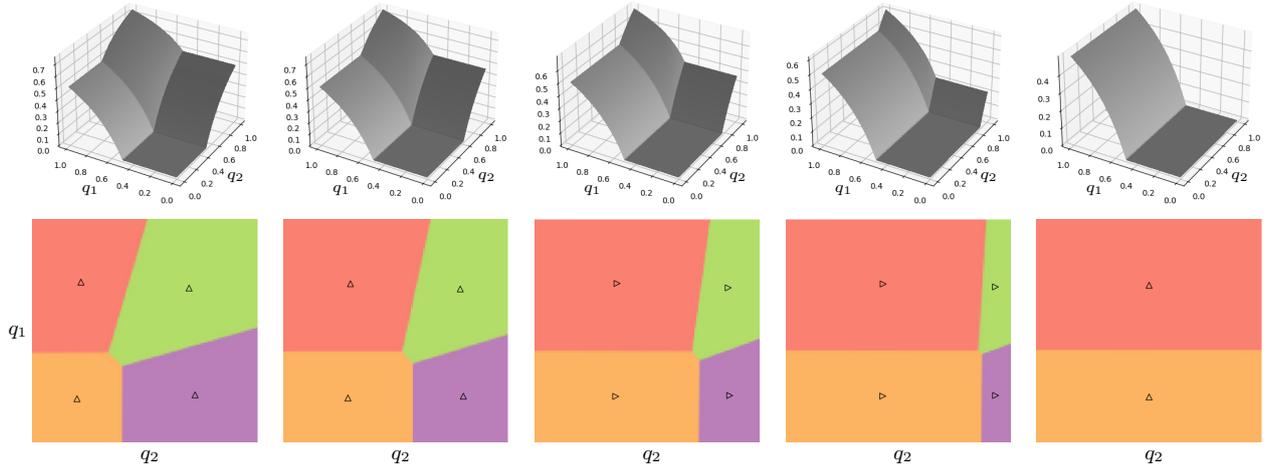

\begin{table}[!h]
\begin{center}
\begin{tabular}{ c | c | c | c | c | c }                
$\underline q$ & 0.25 & 0.5 & 0.75 & 1. & 1.25  \\ 
\hline \hline 
$\mathbf{R_{opt}}$ & $\mathbf{0.2999}$ & $\mathbf{0.3457}$ & $\mathbf{0.4074}$ & $\mathbf{0.4853}$ & $\mathbf{0.5687}$\\ \hline                   
$R(1, 1)$ & $0.1111$ & $0.2000$ & $0.3564$ & $0.4757$ & $0.5687$\\ \hline
$R_{Lloyd}$ & $0.2838$ & $0.3105$ & $0.3604$ & $0.4583$ & $0.5687$\\ \hline
$\mathbb{E}(R)$ & $0.2728$ & $0.3146$ & $0.3714$ & $0.4528$ & $0.5508$\\ \hline
$pp (\%)$ & $11.30$ & $10.01$ & $10.01$ & $10.98$ & $11.86$
\end{tabular}
\caption{Monopolist problem with unit demand, prices $p_1=1$, $p_2=1.25$, upper quality bound $\overline{q}=2$ for different lower quality bounds $\underline q = 0.25, 0.5, 0.75, 1., 1.25$ (first row) under selected information policies. \label{tab:obj_unit_demand_qminus}}
\end{center} \end{table}
\begin{figure}[h]
\resizebox{\linewidth}{!}{
\tikzstyle{frame} = [line width=1.8pt, draw=red,inner sep=0.01em]
\begin{tikzpicture}
\begin{scope}[scale=1.0]
\begin{scope}[shift={(0,-0.)}]
\node[anchor=south west] at (0.,0.) 
{\includegraphics[width=0.2\linewidth]{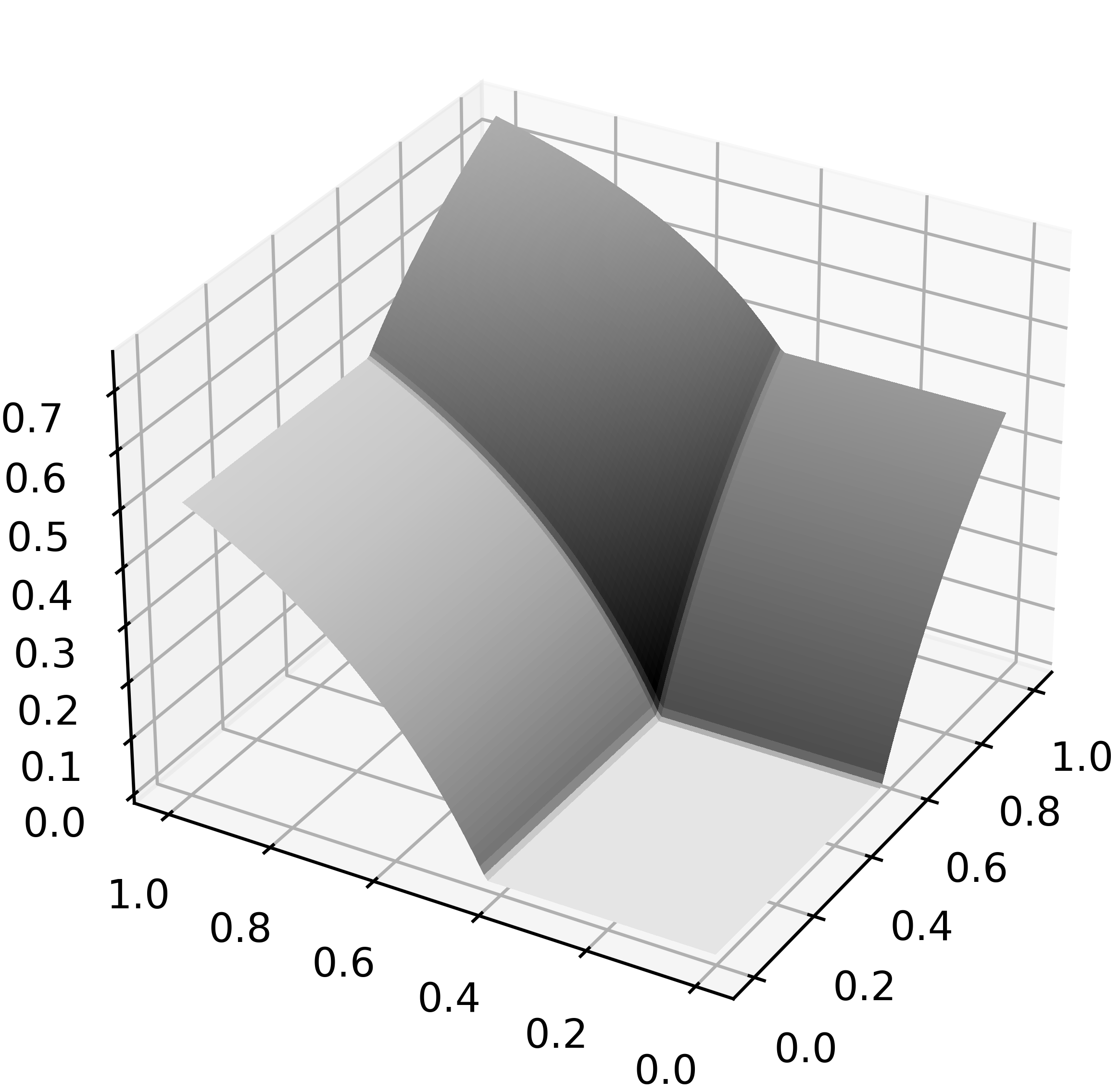}};
\node[anchor=south west] at (3.8,0.) 
{\includegraphics[width=0.2\linewidth]{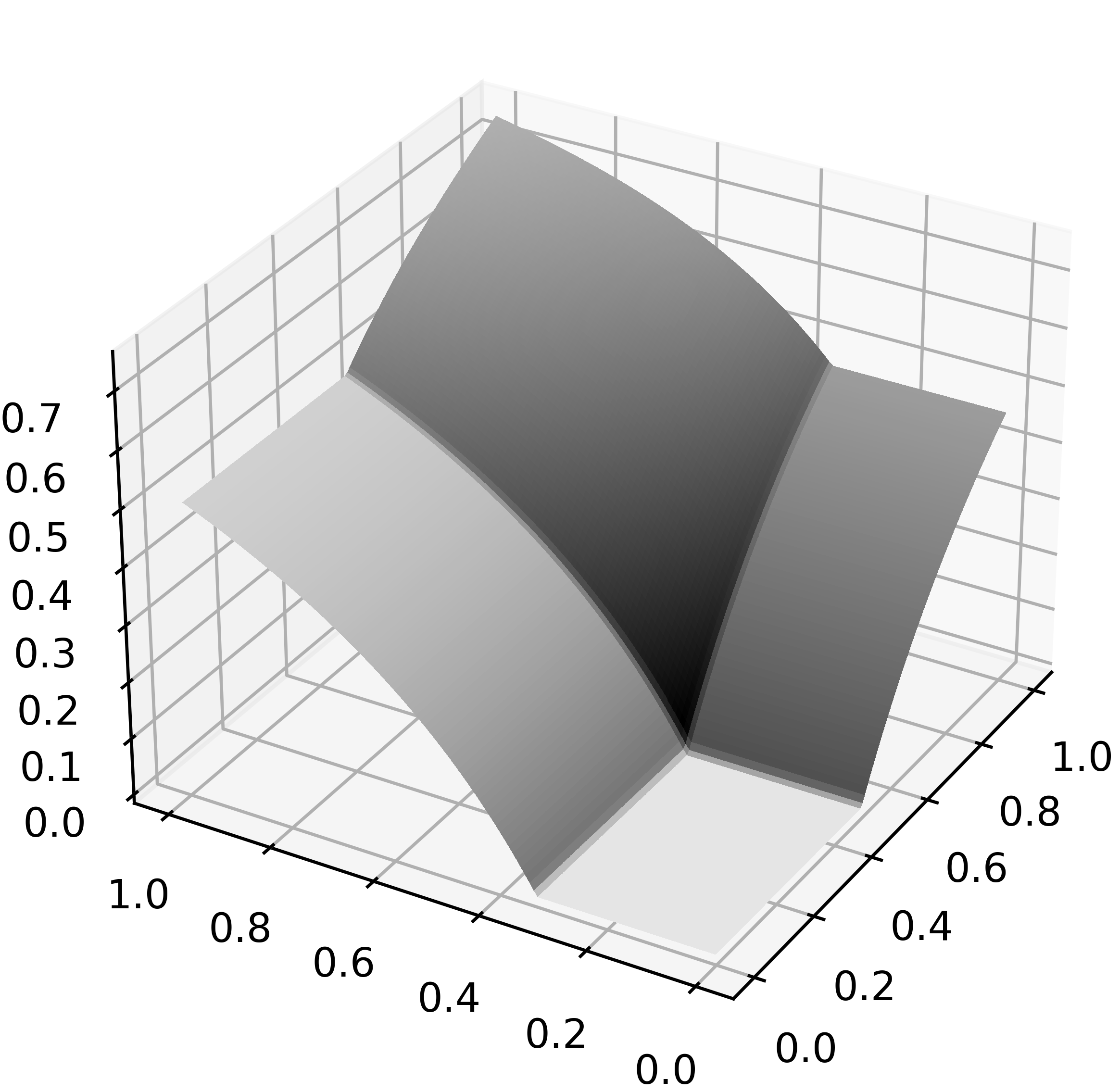}};
\node[anchor=south west] at (7.6,0.)
{\includegraphics[width=0.2\linewidth]{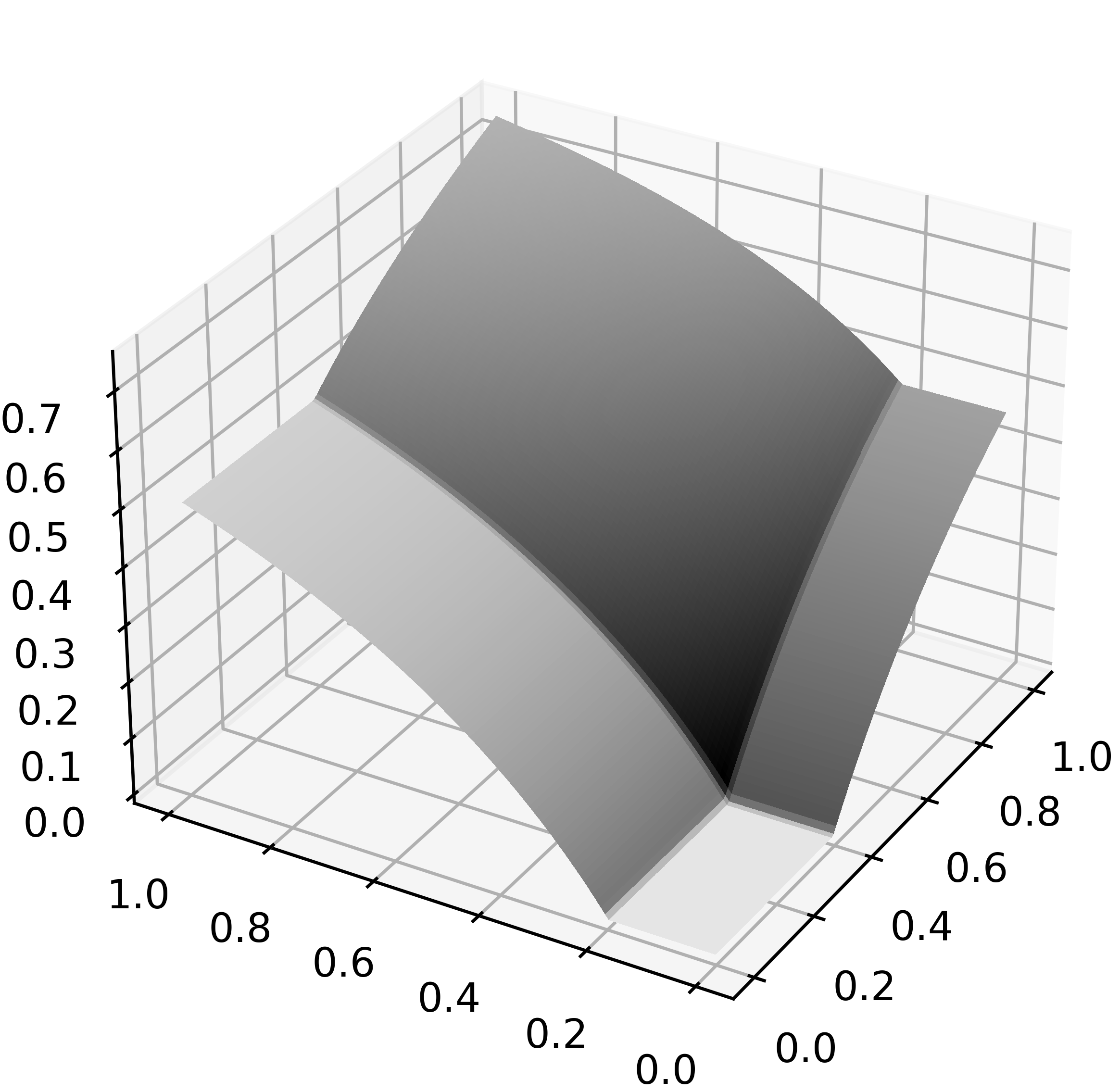}};
\node[anchor=south west] at (11.4, 0.)
{\includegraphics[width=0.2\linewidth]{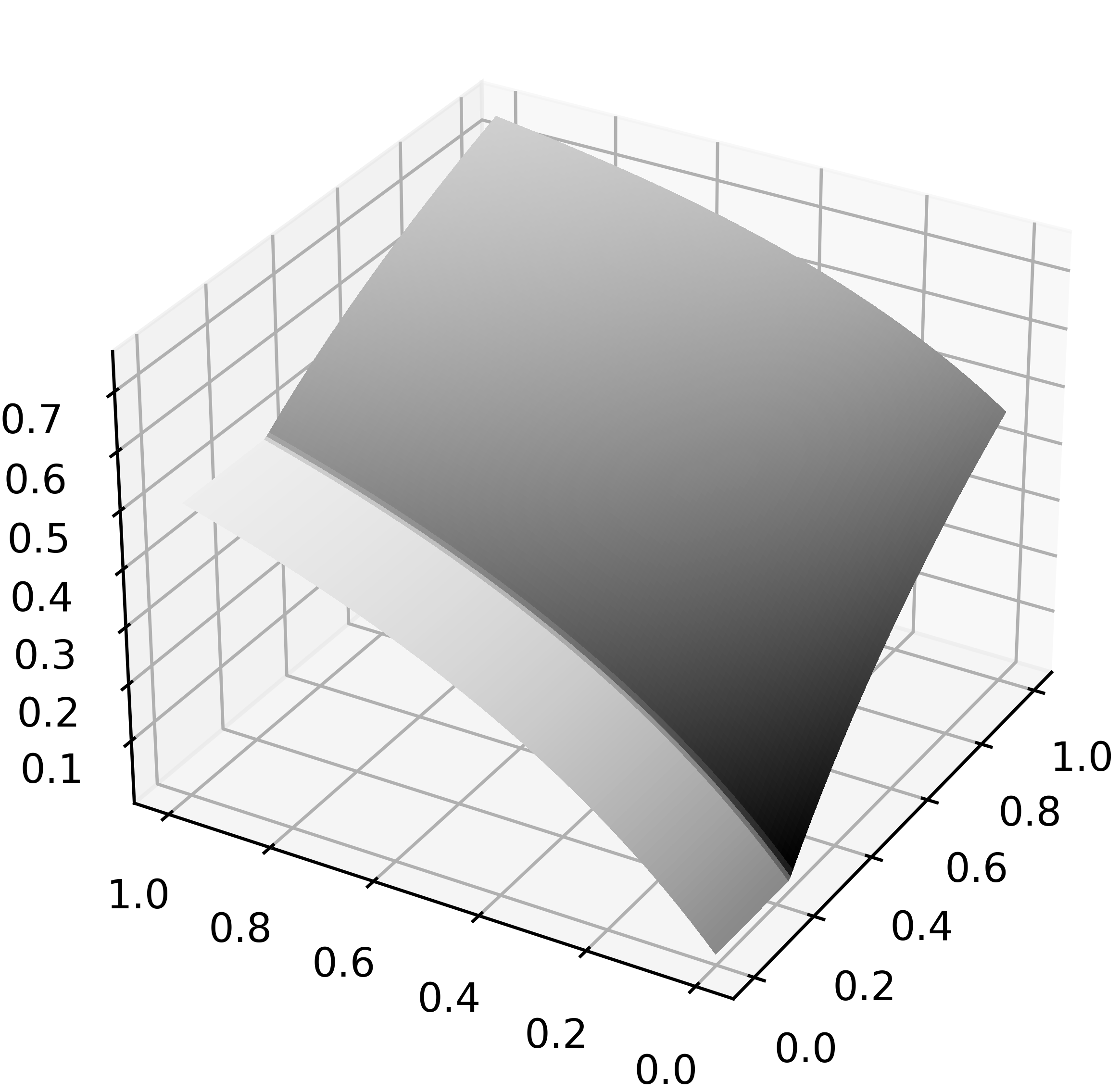}};
\node[anchor=south west] at (15.2,0.)
{\includegraphics[width=0.2\linewidth]{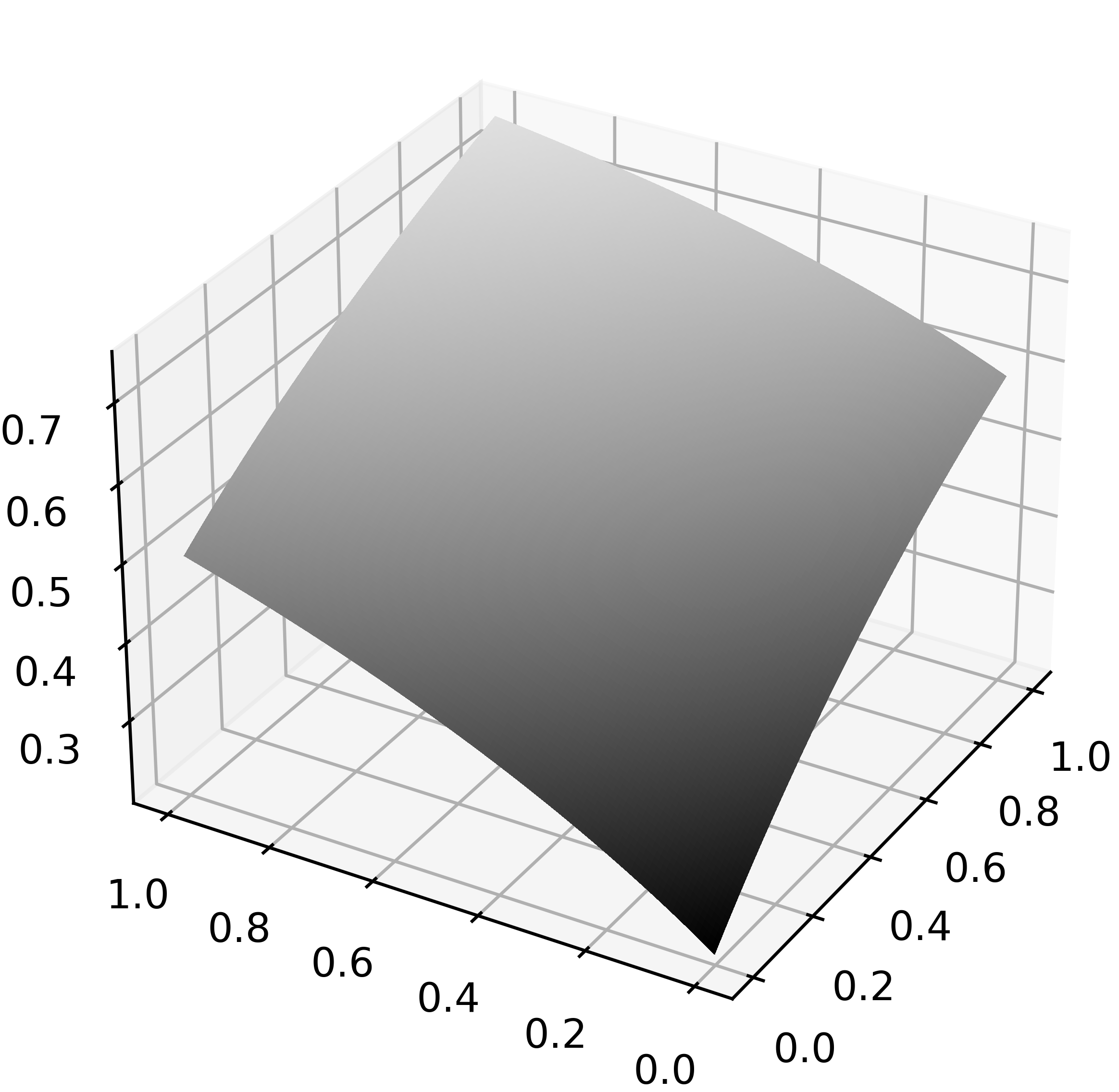}};
\node at (1.0,0.35) {\scriptsize{$q_1$}};
\node at (3.2,0.55) {\scriptsize{$q_2$}};
\node at (4.8,0.35) {\scriptsize{$q_1$}};
\node at (7.0,0.55) {\scriptsize{$q_2$}};
\node at (8.6,0.35) {\scriptsize{$q_1$}};
\node at (10.8,0.55) {\scriptsize{$q_2$}};
\node at (12.4,0.35) {\scriptsize{$q_1$}};
\node at (14.6,0.55) {\scriptsize{$q_2$}};
\node at (16.2,0.35) {\scriptsize{$q_1$}};
\node at (18.4,0.55) {\scriptsize{$q_2$}};
\end{scope}

\begin{scope}[shift={(0,-3.62)}]
\node[anchor=south west] at (0.,0.) 
{\includegraphics[width=0.2\linewidth]{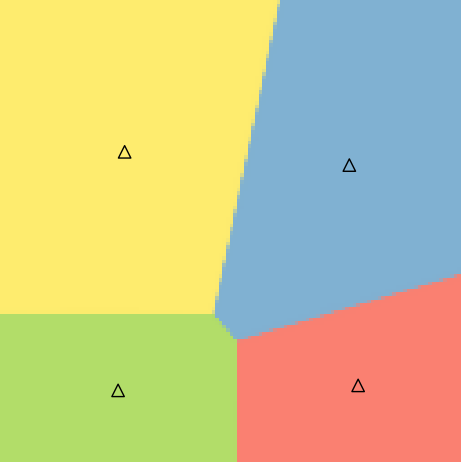}};
\node[anchor=south west] at (3.8,0.) 
{\includegraphics[width=0.2\linewidth]{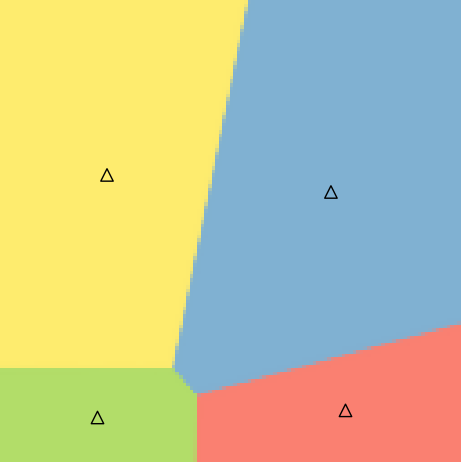}};
\node[anchor=south west] at (7.6,0.)
{\includegraphics[width=0.2\linewidth]{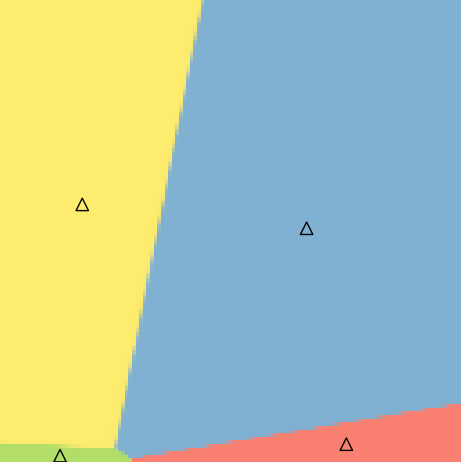}};
\node[anchor=south west] at (11.4, 0.)
{\includegraphics[width=0.2\linewidth]{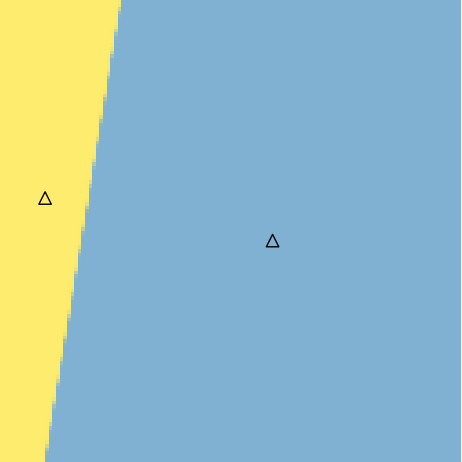}};
\node[anchor=south west] at (15.2,0.)
{\includegraphics[width=0.2\linewidth]{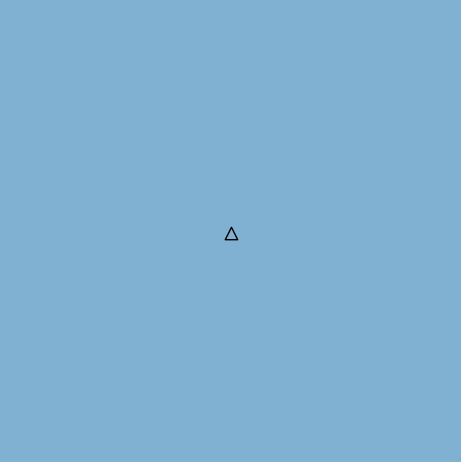}};
\node at (-0.1,1.8) {\small{$q_1$}};
\node at (1.9,-0.1) {\small{$q_2$}};
\node at (5.7,-0.1) {\small{$q_2$}};
\node at (9.5,-0.1) {\small{$q_2$}};
\node at (13.3,-0.1) {\small{$q_2$}};
\node at (17.1,-0.1) {\small{$q_2$}};
\end{scope}
\end{scope}
\end{tikzpicture}
}
\caption{ Optimal configurations for the monopolist's problem with unit demand, prices $p_1=1$ and $p_2=1.25$, upper quality bound $\overline{q}=2$ and, from left to right, lower quality bounds $\underline{q}=0.25,0.5,0.75,1,1.25$ (second row). The respective revenue function $R$ (defined above) is plotted in the first row.}
\label{fig:infodisc_unit_demand_qminus}
\end{figure}
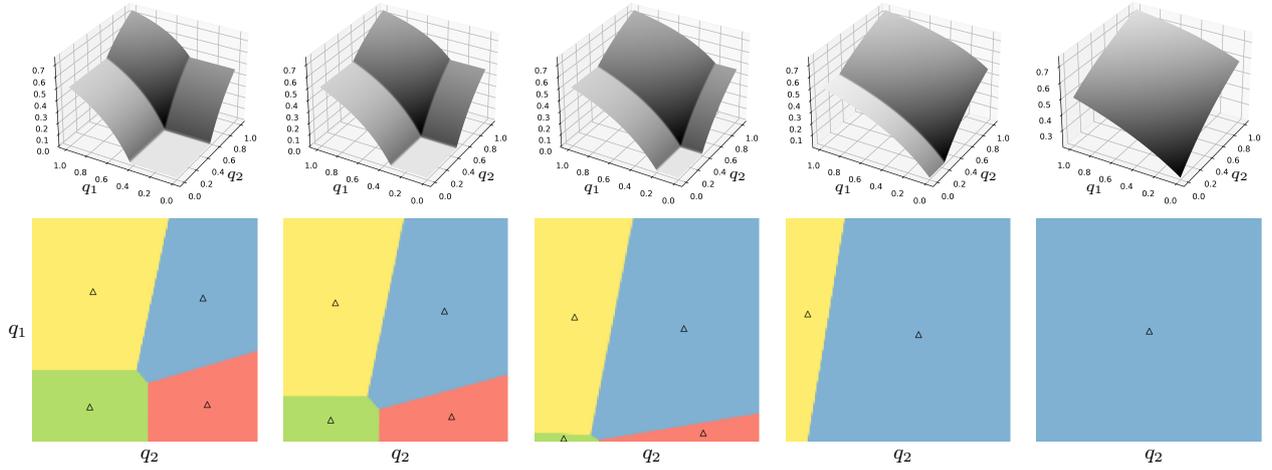

\paragraph{Additive valuations.}
Figure~\ref{fig:infodisc_additive_valuations_sel} illustrates optimal information policies induced by Laguerre diagrams if the buyer has additive valuations. The optimal Laguerre diagrams become richer, with up to seven cells. Moreover, there is significant variation in the shape of the diagrams as the bundling surplus/discount varies. As Tables~\ref{tab:additive_valuations_neg_delta} and \ref{tab:additive_valuations_pos_delta} show, there is again a significant benefit to only partially revealing information compared to fully revealing the qualities.
\begin{table}[!h]
\begin{center}
\begin{tabular}{ c | c | c | c | c | c | c | c | c | c }                
$\Delta$ & -1 & -0.875 & -0.75 & -0.625 & -0.5 & -0.375 & -0.25 & -0.125 & 0  \\
\hline\hline
$\mathbf{R_{opt}}$ & $\mathbf{0.6577}$ & $\mathbf{0.5914}$ & $\mathbf{0.5215}$ & $\mathbf{0.4602}$ & $\mathbf{0.4196}$ & $\mathbf{0.3944}$ & $\mathbf{0.3754}$ & $\mathbf{0.3582}$ & $\mathbf{0.3432}$\\ \hline                     
$R(1, 1)$ & $0.5000$ & $0.4307$ & $0.3516$ & $0.2686$ & $0.1875$ & $0.1143$ & $0.0547$ & $0.0146$ & $0.0000$\\ \hline
$R_{Lloyd}$ & $0.6528$ & $0.5868$ & $0.4824$ & $0.3843$ & $0.3510$ & $0.3433$ & $0.3533$ & $0.3436$ & $0.3333$\\ \hline
$\mathbb{E}(R)$ & $0.6417$ & $0.5196$ & $0.4578$ & $0.4138$ & $0.3803$ & $0.3544$ & $0.3343$ & $0.3188$ & $0.3069$\\ \hline
$pp(\%)$ & $2.49$ & $13.82$ & $13.91$ & $11.21$ & $10.33$ & $11.29$ & $12.29$ & $12.36$ & $11.83$
\end{tabular}
\caption{Monopolist's problem with additive valuations for different bundling discounts $\Delta$ (first row), quality boundaries $\underline{q}=0$ and $\overline{q}=2$ and prices $p_1=p_2=1$, under selected information policies.}\label{tab:additive_valuations_neg_delta}
\end{center} \end{table}
\begin{table}[!h]
\begin{center}
\begin{tabular}{ c | c | c | c | c | c | c }                
$\Delta$ & 0.125 & 0.25 & 0.375 & 0.5 & 0.625 & 0.75 \\
\hline\hline
$\mathbf{R_{opt}}$ & $\mathbf{0.3303}$ & $\mathbf{0.3205}$ & $\mathbf{0.3154}$ & $\mathbf{0.3153}$ & $\mathbf{0.3153}$ & $\mathbf{0.3153}$ \\ \hline                     
$R(1, 1)$ & $0.0000$ & $0.0000$ & $0.0000$ & $0.0000$ & $0.0000$ & $0.0000$\\ \hline
$R_{Lloyd}$ & $0.3231$ & $0.3142$ & $0.3079$ & $0.3056$ & $0.3056$ & $0.3056$\\ \hline
$\mathbb{E}(R)$ & $0.2979$ & $0.2916$ & $0.2875$ & $0.2851$ & $0.2839$ & $0.2834$ \\ \hline
$pp (\%)$ & $10.88$ & $9.91$ & $9.70$ & $10.59$ & $11.06$ & $11.26$
\end{tabular}
\caption{Monopolist's problem with additive valuations for different bundling surcharges $\Delta$ (first row), quality boundaries $\underline{q}=0$ and $\overline{q}=2$ and prices $p_1=p_2=1$ under selected information policies.}\label{tab:additive_valuations_pos_delta}
\end{center} 
\end{table}
\begin{figure}[ht]
\resizebox{\linewidth}{!}{
\tikzstyle{frame} = [line width=1.8pt, draw=red,inner sep=0.01em]
\begin{tikzpicture}
\begin{scope}[scale=1.0]
\begin{scope}[shift={(0,-0.)}]
\node[anchor=south west] at (0.,0.) 
{\includegraphics[width=0.2\linewidth]{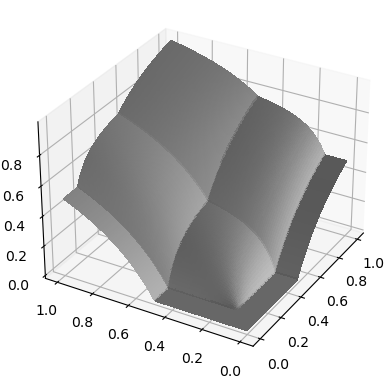}};
\node[anchor=south west] at (3.8,0.) 
{\includegraphics[width=0.2\linewidth]{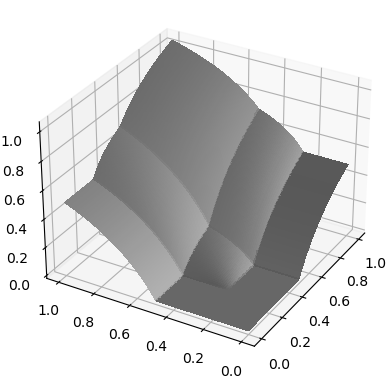}};
\node[anchor=south west] at (7.6,0.)
{\includegraphics[width=0.2\linewidth]{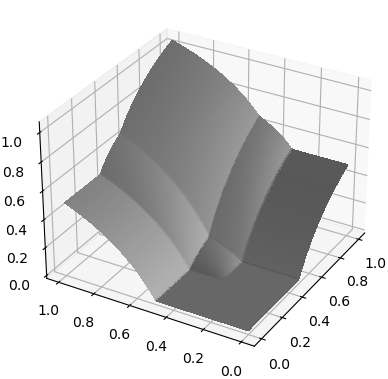}};
\node[anchor=south west] at (11.4, 0.)
{\includegraphics[width=0.2\linewidth]{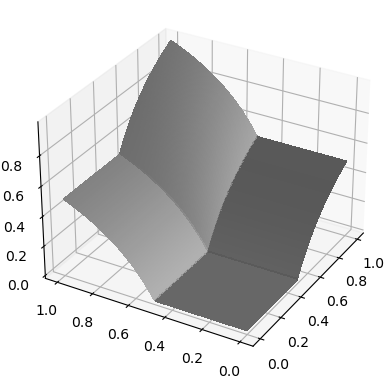}};
\node[anchor=south west] at (15.2,0.)
{\includegraphics[width=0.2\linewidth]{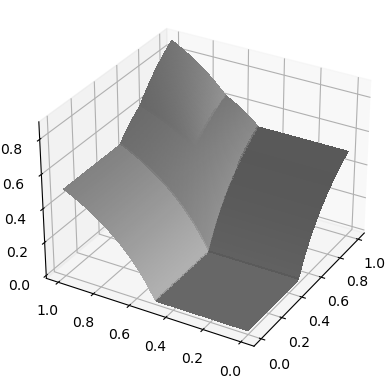}};
\node at (1.0,0.35) {\scriptsize{$q_1$}};
\node at (3.2,0.55) {\scriptsize{$q_2$}};
\node at (4.8,0.35) {\scriptsize{$q_1$}};
\node at (7.0,0.55) {\scriptsize{$q_2$}};
\node at (8.6,0.35) {\scriptsize{$q_1$}};
\node at (10.8,0.55) {\scriptsize{$q_2$}};
\node at (12.4,0.35) {\scriptsize{$q_1$}};
\node at (14.6,0.55) {\scriptsize{$q_2$}};
\node at (16.2,0.35) {\scriptsize{$q_1$}};
\node at (18.4,0.55) {\scriptsize{$q_2$}};
\end{scope}

\begin{scope}[shift={(0,-3.62)}]
\node[anchor=south west] at (0.,0.) 
{\includegraphics[width=0.2\linewidth]{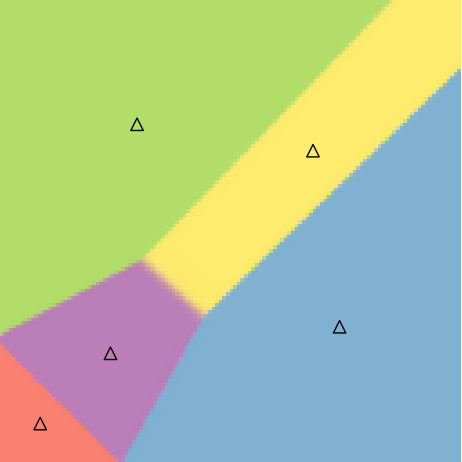}};
\node[anchor=south west] at (3.8,0.) 
{\includegraphics[width=0.2\linewidth]{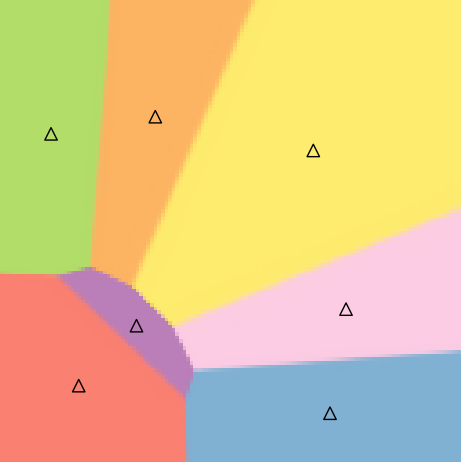}};
\node[anchor=south west] at (7.6,0.)
{\includegraphics[width=0.2\linewidth]{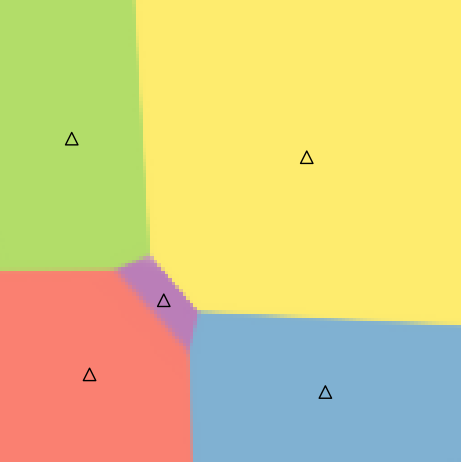}};
\node[anchor=south west] at (11.4, 0.)
{\includegraphics[width=0.2\linewidth]{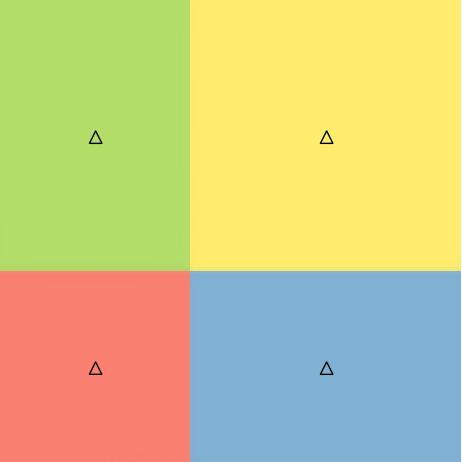}};
\node[anchor=south west] at (15.2,0.)
{\includegraphics[width=0.2\linewidth]{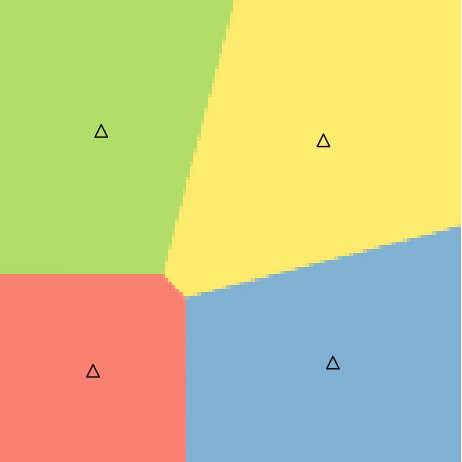}};
\node at (-0.1,1.8) {\small{$q_1$}};
\node at (1.9,-0.1) {\small{$q_2$}};
\node at (5.7,-0.1) {\small{$q_2$}};
\node at (9.5,-0.1) {\small{$q_2$}};
\node at (13.3,-0.1) {\small{$q_2$}};
\node at (17.1,-0.1) {\small{$q_2$}};
\end{scope}
\end{scope}
\end{tikzpicture}
}
\caption{Optimal configurations for the monopolist's problem with additive valuations, with prices $p_1=p_2=1$, quality bounds $\underline{q}=0$ and $\overline{q}=2$ and, from left to right, bundling surcharges/discounts $\Delta=-0.75, -0.5, -0.375, 0, 0.375$ (second row). The respective revenue function $R$ (defined above) is plotted in the first row.}
\label{fig:infodisc_additive_valuations_sel}
\end{figure}
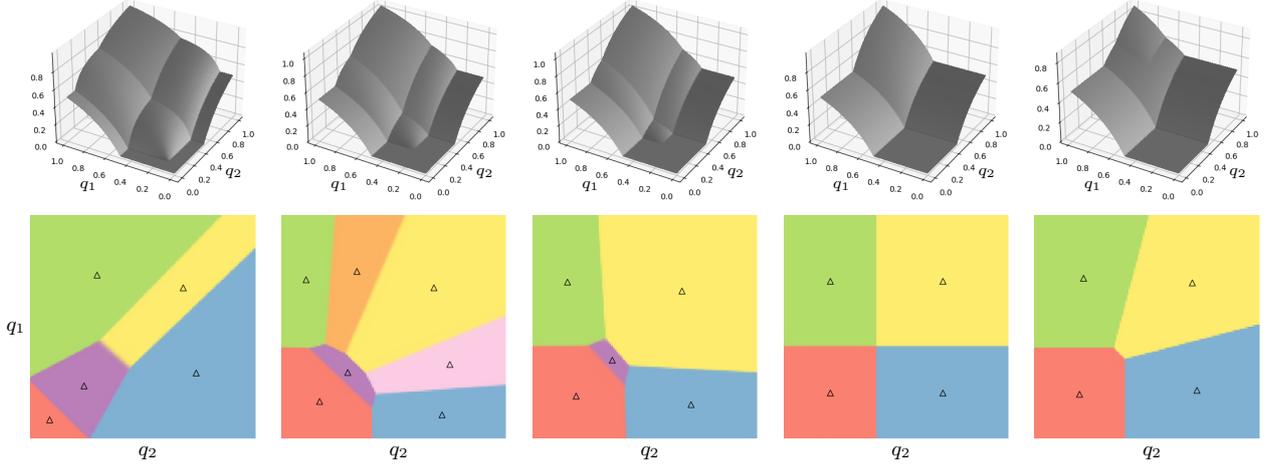

\begin{figure}
\begin{center}
\resizebox{0.85\linewidth}{!}{
\tikzstyle{frame} = [line width=1.8pt, draw=red,inner sep=0.01em]
\begin{tikzpicture}
\begin{scope}[scale=1.0]

\begin{scope}[shift={(0,-11.02)}]
\node[anchor=south west] at (0.,0.) 
{\includegraphics[width=0.3\linewidth]{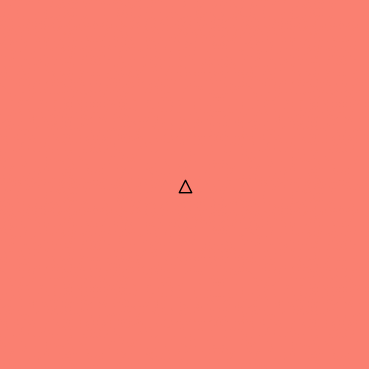}};
\node[anchor=south west] at (5.3,0.) 
{\includegraphics[width=0.3\linewidth]{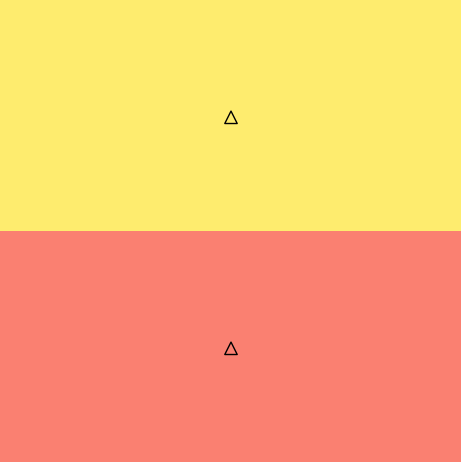}};
\node[anchor=south west] at (10.6,0.) 
{\includegraphics[width=0.3\linewidth]{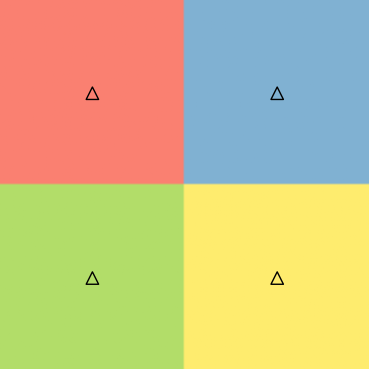}};
\node[anchor=south west] at (15.9,0.) 
{\includegraphics[width=0.3\linewidth]{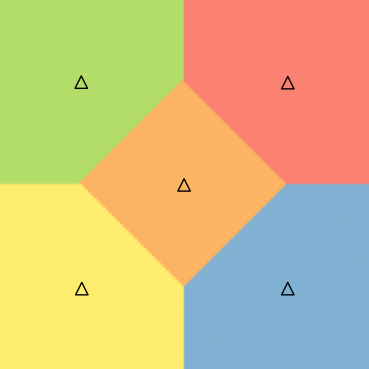}};
\node[anchor=south west] at (21.2,0.) 
{\includegraphics[width=0.3\linewidth]{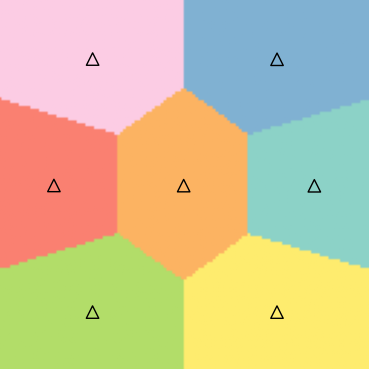}};
\end{scope}
\end{scope}
\end{tikzpicture}
}
\end{center}
\caption{Plot of the Lloyd's centroidal power diagrams for $n=1, 2, 4, 5, 7$ cells from left to right. The sites coincide with the barycenters of the cells.}
\label{fig:RevenuePlot2}
\end{figure} 

\subsubsection*{Acknowledgement} 
This work was supported by the 
German Research Foundation (DFG)
via Germany's Excellence Strategy project 390685813 
-- Hausdorff Center for Mathematics.
\bibliographystyle{abbrv}     
\bibliography{bibliography}

\end{document}